\documentclass[a4paper,11pt,oneside]{amsart}
\usepackage[active]{srcltx} \usepackage[hypertex]{hyperref}

\newtheorem{thm}{Theorem}[section]

\newtheorem{cor}[thm]{Corollary}
\newtheorem{lem}[thm]{Lemma}
\newtheorem{prop}[thm]{Proposition}
\newtheorem{remark}[thm]{Remark}

\def\OO{{\mathcal O}}
\def\FF{{\mathcal F}}

\def\div{{\textsf {div}}}

\def\grad{{\textsf {grad}}}

\begin{document}

\title[Dynamics of Levi-flat]{Topology and dynamics of Levi-flats in surfaces of general type}

\author{Bertrand Deroin \& Christophe Dupont}

\dedicatory{To the memory of Marco Brunella}

\maketitle

\markright{ \today}

\begin{abstract}  
We focus on the topology and dynamics of minimal sets and Levi-flats in surfaces of general type. Our method relies on the ergodic theory of Riemann surfaces laminations: we use harmonic measures and Lyapunov exponents. Our first result establishes that  minimal sets have large Hausdorff dimension when a leaf is simply connected.
Our second result shows that the class of Anosov Levi-flats does not occur in surfaces of general type. In particular, by using rigidity results, we obtain that Levi-flats are not be diffeomorphic to $\Gamma\backslash G$, where $G$ is either $\mathrm{Sol}$ or $ \mathrm{PSL}(2,\mathbb R)$ and $\Gamma$ is a cocompact lattice in $G$. 


\end{abstract}

\vskip .3 cm 

\footnotesize
\noindent \emph{Key Words} : Riemann surface laminations, surfaces of general type, harmonic measure, Lyapunov exponent.
\vskip .1 cm 
\noindent \emph{2010 AMS} : 37F75, 32V40, 37A50, 57R30, 14J29. 


\normalsize 
\section{Introduction}

In a complex surface $S$, a Levi-flat is a smooth real hypersurface $M$ whose Cauchy-Riemann distribution $TM\cap iTM$ is integrable. More generally, a minimal set $M$ is a closed set consisting of leaves of a singular holomorphic foliation on $S$, each of them being dense in $M$. Classical examples of  minimal sets and Levi-flats appear e.g. in flat $\mathbb {CP}^1$-bundles (saturated of the limit set of the monodromy), in complex tori (linear hypersurface) and in fibrations by holomorphic curves (preimage of a simple closed curve in the base). There are other examples in elliptic surfaces \cite{Ne} and Kummer surfaces \cite{Oh}.  Levi-flats in non K\"ahlerian surfaces appear e.g. in Hopf surfaces (Reeb components) and in Inoue surfaces (hyperbolic torus bundles). 

In the present work we study minimal sets and Levi-flats in surfaces of general type. These complex surfaces generalize hyperbolic Riemann surfaces. Examples are given by hypersurfaces of degree $\geq 5$ in $\mathbb P^3$, quotients of unit balls of $\mathbb C^2$, etc. Contrary to the case of the projective plane, there are plenty of examples of minimal sets and Levi-flat hypersurfaces, e.g. by taking ramified coverings of previous examples, see section~\ref{s:examples}. Our goal is to explore their topological and dynamical properties. 

\subsection{Minimal sets} \label{1.1}

The ramified construction mentionned above provides the most general and complicated dynamics that a Kleinian group can produce (Cantor set, quasi-circle, dendrite, etc), in particular, the Hausdorff dimension can be arbitrarily close to $2$. Observe that the topology of the leaves becomes very complicated by using that construction: their fundamental group is infinitely generated. Our first result shows that the dimension must be large if a leaf has no topology.\footnote{Anosov conjectured in the case of the complex projective plane that the leaves of generic singular holomorphic foliations are biholomorphic to the unit disc, appart a countable number which are annulus.}

\begin{thm} \label{t:non discrete}
Let $M$ be a minimal set in a surface of general type with no transverse invariant measure and with a simply connected leaf. Then the Hausdorff dimension of $M$ is $>3$, unless $M$ is an analytic Levi-flat hypersurface. In the latter case, the holonomy pseudo-group  is not discrete and the foliation is ergodic with respect to the Lebesgue measure on $M$. 
\end{thm}  

The proof uses the stationary measure for the leafwise Brownian motions introduced by Garnett~\cite{Garnett}, commonly called harmonic measure. We obtain the following lower bound for the transverse Hausdorff dimension when the holonomy pseudogroup is discrete
\[ \text{Hausdorff dimension} \geq  \text{Kaimanovich entropy} / \vert \text{Lyapunov exponent} \vert . \]
When the holonomy pseudogroup is not discrete, the minimal set is Levi-flat by combining results of \cite{DK} and \cite{LR}. Recall that the Kaimanovich entropy is the growth of separated Brownian motions \cite{K}, the Lyapunov exponent is the rate at which leaves get closer along a Brownian path \cite{DK}. For discrete random walks, the ratio entropy/exponent is equal to the dimension of the harmonic measure \cite{L2, L4}. We stress that this gives a new method to prove that a pseudo-group is non discrete. Up to now, the only known criterion to ensure such a property is to use explicit elements close to identity \cite{Ghys, LR, Nakai}. 

Let us see how the theorem follows in surfaces of general type. We first prove (see proposition \ref{hyphyp}) that the leaves of minimal sets are hyperbolic, i.e. covered by the unit disc. The arguments rely on the negative curvature of $S$, the foliated adjunction formula and Ahlfors currents. We then endow the leaves with the Poincar\'e metric. On one hand, from the presence of a simply connected leaf, the Kaimanovich entropy is equal to $1$. On the other hand, the Lyapunov exponent belongs to $]-1,0[$. The exponent is indeed very related to the geometry of the surface: it is equal to the cohomological intersection between the harmonic current and the normal bundle of the foliation \cite{Deroin}. In surfaces of general type, the negative curvature of the tangent bundle forces the leaves to stay appart from each other. 

\subsection{Topology of Levi-flats} \label {1.2}

We call a Levi-flat \textit{Anosov} if its CR foliation is topologically conjugated to the weak stable foliation of a $3$-dimensional Anosov flow, see section~\ref{anoso}. Classical examples are the weak stable foliation of the geodesic flow on unitary tangent bundles of negatively curved surfaces, and the weak stable foliation of suspensions of linear hyperbolic diffeomorphisms. Many other constructions hold on graph manifolds and hyperbolic manifolds, see \cite{BF, FW, Goodman, HT}. We prove

\begin{thm} \label{t:anosov}
A $C^1$ Levi-flat in a surface of general type is not Anosov. 
\end{thm}

Our method covers the more general case of Levi-flats with a \emph{point at infinity} (in any complex surface):  we prove that the Lyapunov exponent of their ergodic harmonic measures  is $\leq -1$. A contradiction arises in surfaces of general type since the Lyapunov exponent belongs to $]-1,0]$  (as for minimal sets). A cornerstone to prove that  Anosov flows support a point at infinity is to stretch their trajectories in the hyperbolic uniformizations. The upper bound $\leq -1$ for the exponent is established in the point at infinity context, it can be viewed as a Margulis-Ruelle type inequality.   \\

We are able to use theorem~\ref{t:anosov} to determine the possible geometries carried by atoroidal Levi-flats in surfaces of general type. (Toroidal manifolds are not covered by our approach.) Let us recall the eight $3$-dimensional geometries of Thurston, see \cite{Scott}:
\[ \mathbb R^3 \ , \ \mathbb H^3 \ , \ \mathbb S^3 \ , \ \mathbb S^2 \times \mathbb R \ , \ \mathbb H^2 \times \mathbb R \ , \ \mathrm{Nil} \ ,\ \widetilde {\mathrm {SL}(2,\mathbb R)}  \ , \ \mathrm{Sol}. \] 
Not surprisingly, there exist Levi-flats in surfaces of general type modelled on $\mathbb H^3$ and $\mathbb H^2 \times \mathbb R$: we can use singular fibrations or ramified coverings. The fact that $\widetilde{\mathrm{SL}(2,\mathbb R)}$ does appear is more involved, see proposition \ref{circbun}. Recall that this Lie group  provides non trivial Seifert bundles over surfaces of genus $\geq 2$, see~\cite{Scott}. We precisely construct Levi-flat circle bundles with $|e/\chi|$ arbitrary close to $1/5$, where $e$ is the Euler class and $\chi$ is the Euler characteristic of the base. The other geometries can not be realized:  
 
\begin{thm} \label{t:topology} A $C^2$ Levi-flat in a surface of general type can not carry one of the geometries $\mathbb R^3$, $\mathbb S^3$, $\mathbb S^2 \times \mathbb R$, $\mathrm{Nil}$, ${\mathrm {PSL}(2,\mathbb R)} $ or $\mathrm{Sol}$. \end{thm}

We also observe (see section \ref{NOVI}) that Levi-flats must have a trivial second fundamental group. We adopt the terminology: a $3$-manifold carries the geometry of $\mathrm{PSL}(2,\mathbb R)$ if it is a quotient of this group by a lattice. Up to finite coverings, those manifolds are unitary tangent bundles of orientable closed compact surfaces of genus $\geq 2$. 

The proof of theorem \ref{t:topology} relies on  topological and dynamical methods. The first four cases are ruled out by the exponential growth of the fundamental group:  that property comes from the hyperbolicity and the tautness of the CR foliation, we conclude by applying classical Novikov's theory. To eliminate the remaining two cases, i.e. $ \mathrm{PSL}(2,\mathbb R)$ and  $\mathrm{Sol}$, we establish for such Levi-flats the Anosov property, and we apply theorem~\ref{t:anosov}. That step relies on deep theorems by Thurston~\cite{Thurston}, Ghys-Sergiescu~\cite{GS} and Matsumoto~\cite{Matsu}. We note that our results (subsections \ref{1.1} and \ref{1.2}) actually hold for immersed minimal sets and immersed Levi-flats, we write all the proofs in the embedded case for sake of simplicity. 

\subsection{Acknowledgements}

We thank \'E. Ghys for sharing his views on the theory of $2$-dimensional foliations of $3$-manifolds, and S. Boucksom for discussions about surfaces of general type. Part of this work was motivated by questions asked by M. Brunella, we dedicate this work to him. 

The authors are  supported by the ANR projects 08-JCJC-0130-01, 09-BLAN-0116 for B.D. and 07-JCJC-0006-01 for C.D. The work was initiated during a common visit of the authors in IMPA,  made possible with the help of ANR and France-Br\'esil cooperation. We thank the IMPA and the Mittag-Leffler Institute for the very nice working conditions  offered to us. 


\section{Definitions and examples} \label{s:examples}

We recall the definition of surfaces of general type. We also present examples of  minimal sets and Levi-flat hypersurfaces in those surfaces, with various dynamical and topological properties. 

\subsection{Surface of general type}\label{sgt1}
  
 We refer to the books \cite{BPV}, \cite{F} and \cite{Ko} for a general account. Let $S$ denote a compact complex surface and $K_S$ its canonical line bundle. The surface $S$ is of \emph{general type} if the dimension of the space of holomorphic sections of $n K_S$ grows like $n^2$. Surfaces of general type are algebraic, and the general type property is a birational invariant, stable by blowup. If $S$ is minimal, the pluricanonical map associated to the linear system $\vert n K_S \vert$ is birational onto its image for $n$ large. Basic examples are given by products of curves of genus larger than $2$, general complete intersections in projective spaces (e.g. hypersurfaces of degree  $\geq 5$ in $\mathbb P^3$) and cocompact quotients of the unit ball in $\mathbb C^2$.

Other important examples are obtained by ramified covering. Let us explain the construction starting with an algebraic surface $X$ (see \cite{BPV}, section I.17), we shall use it in the next subsections. Let $D \subset X$ be a smooth reduced effective divisor and $s$ be a holomorphic section of $\OO(D)$ such that $D = \{ s = 0 \}$. Let $d \geq 1$ and assume that $\OO(D) = d E$ for some line bundle $E$. Then $S := \{  z \in E \, , \, z^d = s \}$ is a smooth algebraic surface whose canonical bundle is the pull-back of $K_X + (d-1) E$ by the projection $\pi : S \to X$. In particular, $S$ is of general type if $K_X + (d-1) E$ is  ample.

\subsection{Minimal sets} 

Let $\FF$ be a holomorphic foliation on $S$ and $\mathcal E$ denote its possible finite singular set. A \emph{minimal set} $M$ of $\FF$ is a compact union of leaves such that $M$ does not intersect  $\mathcal E$ and every leaf of $M$ is dense in $M$. 

\begin{prop}\label{smallH} 
There exist  minimal sets in surfaces of general type with arbitrarily small transverse Hausdorff dimension. 
\end{prop}

\begin{proof} Let $M$ be a minimal set in an algebraic surface $X$ with arbitrarily small transverse Hausdorff dimension. For instance $X$  can be a flat $\mathbb P^1$-bundle whose monodromy is a Kleinian group whose limit set has small Hausdorff dimension. Let $E$ be a line bundle such that $2E$ is very ample and $K_X + E$ is ample, and let $s$ be a generic section of $2E$. Then, the zero divisor $(s)_0$ is transverse to the foliation at each point of $M$. Let $c: S \rightarrow X$ be the double covering defined in subsection~\ref{sgt1} and $M' = c^{-1} (M)$. This is a minimal set in $S$ with same Hausdorff dimension than $M$. \end{proof} 

Observe that the same arguments show that any  minimal set with a transverse structure modelled on the limit set of a Kleinian group does exist in surfaces of general type. Moreover, that allows to construct minimal Levi-flat hypersurfaces in surfaces of general type with transverse invariant measures. This can be done by considering a ramified covering of a flat $\mathbb P^1$-bundle whose monodromy is dense in $\mathrm{PSO} (2,\mathbb R) \subset \mathrm{PSL} (2,\mathbb C)$. 

\subsection{Levi-flat hypersurface} 

A real hypersurface $M$ of class $C^1$ in a complex surface $S$ is \emph{Levi-flat} if the distribution $p \mapsto T_p M \cap i T_p M$ is integrable in the sense of Frobenius. We denote by $\FF$ its Cauchy-Riemann (CR) foliation. If $M$ is of class $C^k$ ($k\geq 1$), then the CR foliation is also of class $C^k$, see \cite{BF}. We do not assume that the leaves are dense in $M$ neither that $M$ is a minimal set (a priori its CR foliation can not be extended to a singular holomorphic foliation on $S$). We call the Levi-flat $M$ \emph{minimal} when every leaf is dense in $M$.

\begin{prop}\label{circbun}
There exists a Levi-flat hypersurface in a surface of general type which is diffeomorphic to a circle bundle such that the ratio Euler class/Euler characteristic is arbitrarily close to $1/5$.
\end{prop}

We do not know any example of a Levi-flat circle bundle in a surface of general type such that  $\text{Euler class}/\text{Euler characteristic}\geq 1/5$. 




\begin{proof} Let $\Gamma \subset \mathrm{PSL} (2,\mathbb R)$ be a cocompact and torsion free Fuchsian group. Let $X$ be the quotient of $\mathbb H\times \mathbb P^1$ by the action of $\Gamma$ given by $\gamma (p,z) = (\gamma (p) , \gamma (z) )$. This is a flat $\mathbb P^1$-bundle over $\Sigma:= \Gamma \backslash \mathbb H$.  We denote by $M\subset X$ the Levi-flat hypersurface which is the quotient of $\mathbb H \times \mathbb P^1 (\mathbb R)$, this circle bundle is diffeomorphic to the unitary tangent bundle of $\Sigma = \Gamma \backslash \mathbb H$. 

We consider the section $\sigma$ which is the quotient of the curve $\{  (p,p)\in \mathbb H \times \mathbb H \ |\ p\in \mathbb H\}$. Observe that $\sigma$ is transverse to the $\mathbb P^1$-fibration and to the horizontal foliation associated to the flat connexion. Hence, we have $\sigma^2 = \chi$, where $\chi= 2-2g$ is the Euler characteristic of $\Sigma$, see~\cite{LorayMarin}. It is a classical fact that a $\mathbb P^1$-bundle having a section of even self-intersection is indeed diffeomorphic to the product bundle. Hence, we have $H^2 (X,\mathbb Z) = \mathbb Z [\sigma] + \mathbb Z [\varphi]$ where $\varphi$ is a fiber of the $\mathbb P^1$-fibration. Adjunction formula  then yields $[K_X] = -2 [\sigma] $. In order to perform the ramified covering construction, we need the following lemma.

\begin{lem} \label{l:ampleness}
Let $E' := 6\sigma + \sum _{i=1}^{4(1-2\chi)} \varphi_i  $, where the $\varphi_i$'s  are disjoint fibers of $X$. Then there exists a line bundle $E$ such that $2E = E'$. Moreover, $E'$ is very ample, and $E + K_X$ is ample. 
\end{lem}

Let $s^0$ be a non zero section of $E'$ vanishing on $6\sigma + \sum _i \varphi_i$. Because $E'$ is very ample, there exists a small deformation $s$ of $s^0$ which is smooth. This section intersects the Levi-flat $M$ in a link isotopic to a family consisting of $4(-2\chi+1)$ distinct fibers $\varphi_i$ of the fibration. We denote by $c: S\rightarrow X$ the double covering defined by $s$ as indicated in subsection \ref{sgt1}.

We now claim that $M'= c^{-1} (M)$ is a circle bundle which is the pull-back (as a circle bundle) of $M$ by a double ramified covering $\Sigma' \rightarrow \Sigma$. To verify this, we begin by considering an isotopy which sends the $\mathbb P^1 (\mathbb R)$-fibration of $M$ to a fibration $F$  having the link $(s)_0\cap M$ as sections.  Then, the pull-back of this fibration by $c$ is a Seifert fibration in $M'= c^{-1} (M)$. Since the fibers of $F$ which are far away from $(s)_0 \cap M$ bound discs in $X\setminus (s)_0$ (this is true for $6 \sigma + \sum_i \varphi_i$), this Seifert fibration is indeed a genuine fibration, and our claim follows. 

Finally, the Euler class of the bundle structure on $M'$ is twice the Euler class of the base, hence $eu'= 2 eu = 2\chi$. Moreover the Euler characterisic of $\Sigma'$ is $2 \chi + 4(2\chi -1) = 10 \chi - 4$ by Hurwitz formula. That completes the proof of  proposition~\ref{circbun}. \end{proof}

\begin{proof} (of lemma \ref{l:ampleness}) If $\varphi$ is a fixed fiber, the line bundle $E' - (6\sigma + 4(1-2\chi) \varphi )$ has a vanishing Chern class. Hence it belongs to the Jacobian of $X$, and consequently has a square root $E''$, namely $E' - (6\sigma + 4 (1-2\chi) \varphi ) = 2E''$. We then set $E = E''+ 3\sigma + 2(1-2\chi) \varphi$. 

Let us prove that $E + K_X$ is ample. Assuming that $\sigma^2$ is negative, a classical fact is that a line bundle $L$ on $X$ with cohomology class $[L]= a [\sigma] + b [\varphi]$ is ample iff $a>0$ and $\chi a + b >0$. That relies on Nakai-Moishezon criterion, stating that a line bundle is ample iff its intersection with every algebraic curve is positive, and its self-intersection is positive.
Since $[E+K_X]= [\sigma] + 2(1-2\chi) [\varphi]$, we infer that $E+K_X$ is ample. 

To prove the very ampleness of $2E$, we use a criterion of Reider~\cite{Reider},
namely that if a line bundle $L$ is ample, then $4L + K_X$ is very ample. So we look for $L$ such that $2E = 4L+ K_X$.  The existence of such a line bundle $L$ relies on the fact that the line bundle $2E - K_X - 4 (2\sigma + (1-2\chi) \varphi)$ has a vanishing cohomology class. Hence it belongs to the Jacobian, and consequently has a fourth root, namely there exists $F$ such that $4F=2E - K_X - 4 (2\sigma + (1-2\chi) \varphi)$. We just set $L= F+2\sigma + (1-2\chi) \varphi$ to get the desired line bundle. It remains to see that $L$ is ample, this can be done using the previous criterion.
\end{proof}

\section{Harmonic currents}\label{s:harmonic currents}

In this section, $M$ is either a minimal set or a $C^1$ Levi-flat hypersurface in a compact complex surface $S$. We introduce harmonic currents on $M$ with respect to the foliation. These currents generalize the notion of foliated cycles introduced by Sullivan \cite{S} and permit to interpret $M$ as a curve \cite{Ghys2}. Assuming that $S$ is K\"ahler, any harmonic current  has a cohomology class in the Dolbeault cohomology group $H^{1,1} (S,\mathbb C)$ which allows to use intersection theory, see \cite{Deroin,FS1,FS2}. We stress that all the arguments also hold for immersed Levi-flats, we work in the embedded case for simplicity. 

\subsection{Leaves and charts}    

Any leaf of $M$ is an immersed Riemann surface. We call it \emph{elliptic}, \emph{parabolic} or \emph{hyperbolic} depending on its universal covering $\mathbb P^1$, $ \mathbb C$ or $\mathbb D$. The foliation $\FF$ is   \emph{hyperbolic} if every leaf is hyperbolic.  $M$ is covered by finitely many foliated charts $U_j \simeq \mathbb D \times F_i$, with $F_i \subset \mathbb D$, and overlaping as
\begin{equation}\label{rule}
 (z , t)  =  \left( \alpha_{jj'} ( z' , t' )  ,   \beta_{jj'} ( t' ) \right) .   
 \end{equation}
The functions $\alpha, \beta$ are continuous and $\alpha(\cdot,t')$ is holomorphic for every $t'$. For $C^{k}$ Levi-flat we have $F_i =]0,1[$, the functions $\alpha , \beta$ are of class $C^k$ by \cite{BF}, and $\alpha(\cdot,t')$ depends continuously on $t'$ in the $C^\infty$ topology.

\subsection{Calculus} \label{calculus}    
 
Let $\mathcal O_\FF$ (resp. $C^\infty_\FF$) denote the sheaf of  functions  which are leafwise holomorphic (resp. smooth).  Let $A^p_\FF$ (resp. $A^{p,q}_\FF$) denote  the space of leafwise $p$-forms (resp. $(p,q)$ forms). 
A form $\eta \in A^{1,1}_\FF$ is \emph{positive} if  $\eta = c (z,t) \, i \, dz \wedge d\bar z $ locally, where $c \in C^\infty_\FF$ is positive. For $f \in A^0_\FF$ we define $\partial f = \partial_z f \cdot dz$ and $\bar \partial f = \partial_{\bar z} f \cdot d\bar z$, where $\partial_z := {1\over 2} ( \partial_x - i\partial_y)$ and $\partial_{\bar z} := {1\over 2} ( \partial_x + i\partial_y)$. The leafwise differential operator $d : A^p_\FF \to A^{p+1}_\FF$ is then equal to $\partial + \bar \partial$ and $i\partial \bar \partial : A^0_\FF \to A^{1,1}_\FF$ is a real operator. 

Given a metric $m$ on the tangent bundle $T_\FF$ (see below), we denote by $vol_m$ the leafwise volume form and by $\Delta_m = \div \, \grad$ the leafwise Laplacian. They satisfy $2i \partial \bar \partial  = \Delta_m  \cdot vol_m$.

\subsection{Foliated cycles} \label{FC}    

A \textit{foliated cycle} is a continuous linear form $T : A^{1,1}_\FF  \to \mathbb R$ such that 

- $ T (d \varphi)  = 0$ for any  $\varphi \in A^1_\FF$, 

- $T (\eta) > 0$ for any positive form $\eta \in A^{1,1}_\FF$.

The existence of a foliated cycle is a strong condition. There is a correspondance between foliated cycles and transverse invariant measures \cite{S}. 
 Any compact leaf  yields a foliated cycle \emph{via} its current of integration.  Any parabolic leaf produces foliated cycles as follows. Let $f : \mathbb C \to L$ be a uniformization. A metric on the tangent bundle $T_\FF$ being fixed, let $a_r$ denote the area of $D_r := f(\mathbb D_r)$ and $l_r$ denote the length of $\partial D_r$. By an argument due to Ahlfors, there exists a sequence $r_n \to \infty$ such that $l_{r_n} / a_{r_n} \to 0$. The normalized current of integration $A_n := {1 \over a_{r_n}} [D_{r_n}]$ then converges to a closed positive current $A$,  called an \emph{Ahlfors current}.  Those considerations show that $\FF$ is hyperbolic if $M$ has no transverse invariant measure. The converse is false, see subsection \ref{ss:hyperbolicity}.

\subsection{Harmonic currents} \label{HC}

A \textit{harmonic current}  is a continuous linear form $T : A^{1,1}_\FF  \to \mathbb R$ such that 

- $T (\partial \overline{\partial} \varphi)  = 0$ for any $\varphi \in A^0_\FF$, 

- $T (\eta) > 0$ for any positive form $\eta \in A^{1,1}_\FF$.

Hahn-Banach theorem allows to prove the existence of harmonic currents, see \cite{Ghys2}. They can also be constructed from a hyperbolic leaf by using a weighted Ahlfors procedure, see \cite{FS1}. A harmonic current has the following local expression in  $U_j \simeq \mathbb D \times F_i$:
\[ T = \left[ H_j (z,t) \, i  \, dz \wedge d\bar z \right] \otimes \nu_j  ,\]
where $H_j (\cdot,t)$ are positive harmonic functions and $\nu_j$ is a positive measure on $F_i$.  This is a foliated cycle if and only if $H_j(\cdot,t)$ is constant. 

The harmonic current is unique for laminations and singular foliations of $\mathbb P^2$ \cite{FS1}, \cite{FS2}, it is also unique for conformal foliations which do not admit any transverse invariant measure \cite{DK}. Let us see how a harmonic current defines a cohomology class $[T]$ in the Dolbeault cohomology group $H^{1,1}(S , \mathbb C)$. The Bott-Chern cohomology group $H_{BC}^{1,1}(S,\mathbb C)$  is  the quotient of the space of smooth closed $(1,1)$-forms by the space of $\partial \bar \partial$-exact ones (see  \cite{Demailly}, section VI.8). When $S$ is a K\"ahler surface, the natural morphism $H_{BC}^{1,1}(X, \mathbb C) \to H^{1,1}(X, \mathbb C)$ is an isomorphism by the classical $\partial \bar \partial$-lemma. In particular, since a harmonic current belongs to the dual of $H_{BC}^{1,1}(S,\mathbb C)$, it also belongs to the dual of $H^{1,1}(S,\mathbb C)$, which is isomorphic to $H^{1,1}(S,\mathbb C)$ by duality. That allows to define $[T] \in H^{1,1}(S , \mathbb C)$ as desired.

\subsection{Line bundles} 

These are the elements of $H^1(M, \mathcal O^*_\FF)$. They can be defined using the charts by identifying $(z,t,\xi) \in U_j \times \mathbb C$ with $(z',t',\xi') \in U_{j'} \times \mathbb C$ using (\ref{rule}) and the rule $\xi = \gamma_{jj'} (z' , t')  \cdot \xi '$ for some $\gamma_{jj'} \in \mathcal O_\FF^*$.  We shall use the following line bundles on $M$

\begin{enumerate}
\item[-] \emph{Tangent line bundle} $T _\FF$:  $\gamma_{jj'}  := \frac {\partial \alpha_{jj'} }{ \partial{z'}}$.
\item[-] \emph{Canonical line bundle} $K_\FF$:  $\gamma_{jj'}  := ({\partial \alpha_{jj'} \over \partial{z'}} )^{-1}$ (dual of $T_\FF$)
\item[-] \emph{Normal line bundle} $N_\FF$: $\gamma_{jj'}  := {d \beta_{jj'} \over dt'}$. 
\end{enumerate}

The definition of $N_\FF$ concerns Levi-flat hypersurfaces, $\beta$ is in that case (at least) of class $C^1$. The definition of $N_\FF$ for minimal sets uses the ambiant foliation. In any case the transition functions of $N_\FF$ are leafwise constant. We have the following adjunction formula (see \cite{BrunellaCF}, section 1), which also holds for immersed Levi-flats (see \cite{Deroin}, section 3).
\begin{prop}\label{adju}
$K_\FF = {K_S}_{\, \vert \FF} +  N_\FF$. 
\end{prop}

 A metric $h$ on a line bundle $L$ is locally defined by $h_j (z , t ) \vert \xi \vert$, where $h_j \in C^\infty_\FF$ is  positive. If $L$ is defined by $(\gamma_{jj'})$ then  
\begin{equation*}\label{ddss} 
h_{j'} (z' , t' ) = h_j (z , t) \vert \gamma_{jj'} (z' , t' ) \vert . 
\end{equation*}
Metrics always exist, they can be  constructed by using partitions of unity. If $h'$ is another metric on $L$, then $h' = h e^\tau$ for some   function $\tau \in C^\infty_\FF$. 
The \emph{Chern curvature} of $h$ is the form $\omega_h \in \mathcal A^{1,1}_\FF$ locally defined by 
\[  \omega_h := \frac{1}{2 i \pi}  \partial \overline{\partial} \log h_j (z,t) ^2.  \]
If $h' = h e^\tau$ is another metric, then $\omega_{h'} = \omega_h + \frac{1}{i \pi}  \partial \overline{\partial} \tau$. If $T$ is a harmonic current, we define the intersection
\[  T \cdot L  :=   [T] \cdot [L]  = T (\omega _h ) , \] 
the last equality being independent of the metric $h$ on $L$. 

\subsection{Some intersections} 

The \emph{Euler-Poincar\'e characteristic} of  a harmonic current is $\chi(T) := T \cdot T_\FF$.
That definition extends Gauss-Bonnet formula for compact smooth curves. We have the following result, see \cite{Candel}. 
 
\begin {prop}[Candel] \label{cha}
If $\FF$ has a parabolic leaf $L$, then there exists an Alhfors current such that $\chi(A) = 0$.
\end{prop}

The arguments use Gauss-Bonnet formula for open Riemann surfaces: the idea is to perturb the boundary of $A_n$ (see subsection \ref{FC} for the definition) in order to control its geodesic curvature. The following lemma is crucial, the proof uses the Bott connexion on $N_\FF$ (see \cite{CanCon2}, chapter 6), which is flat when restricted to  $\FF$. The intersection $T\cdot N_\FF$ will be interpreted later as a Lyapunov exponent.

\begin{lem}\label{normalbott} $ $
\begin{enumerate}
\item $T\cdot N_\FF = 0$ for every foliated cycle $T$,
\item $C \cdot C = 0$ for every compact leaf $C$.
\end{enumerate}
\end{lem}

\begin{proof}  The first item (see \cite{CLS}, theorem 2 for a related result) follows from the fact that if $T$ is a foliated cycle and $L$ is a line bundle, then  $T\cdot L= T (\omega)$ where $\omega$ is the curvature of any smooth connexion on $L$. Indeed the curvatures of two connexions differ by a closed $2$-form. The second item is a consequence of the first one and the fact that since $C$ is embedded, then $C\cdot C = T \cdot N_C$ where $T$ is the current of integration on $C$ and $N_C$ is the normal bundle.  \end{proof}

We shall need the following proposition. It was proved in \cite{Ghys2}, lemma 4.5 for foliated cycles. 

\begin {prop}\label{blowup1}
 Let $L$ be a line bundle on $M$. Assume that it has a holomorphic section $s$ which does not vanish identically on any leaf. Then  $T \cdot  L  \geq 0$ for any harmonic current $T$.  
\end{prop}

\begin{proof} 
It suffices to construct a smooth metric $|.|$ on $L$ whose curvature is non negative. To do so we construct a smooth leafwise surharmonic function $\varphi : M  \rightarrow \mathbb R$ with logarithmic singularities along $D := \{ s= 0 \}$ corresponding exactly with the multiplicities defined by $s$. We then define the metric by $|s| := e^{\varphi}$.  

Let $\{U_j\}$ be a finite covering of $M$ such that the line bundle and the foliation are trivial on $U_j$. We fix diffeomorphisms $U_j \simeq \mathbb D \times (0,1)$ such that the foliation is horizontal (for a minimal set replace the interval by $F_j$), let also $1_j: U_j \rightarrow L$ be a non vanishing holomorphic section.  Hence $s = s_j 1_j$ for some holomorphic function $s_j : U_j \rightarrow \mathbb C$. Schrinking the open sets $U_j$ and multiplying the sections $1_j$ by small constants, we can suppose that there exists $0<r<1$ such that $|s_j(z,t)| >1 $ if $|z|\geq r$. In particular, $\varphi_j(z,t) :=\log^- |s_j|$ is leafwise surharmonic on $U_j$. Let us regularize that function. Let $\varepsilon := \min \{ r/2, (1-r)/2 \}$ and $K: [0,1]\rightarrow \mathbb R^+$  be a smooth function with support   in $[0, \varepsilon)$ satisfying $\int_{0}^1 K(t) dt = 1$. Let us define  $\varphi_j^{reg}: U_j  \rightarrow \mathbb R$ by   $\varphi_j^{reg} (z,t) := \varphi_j(z,t)$ if $|z| < r/2$ and 
\[  \varphi_j^{reg} (z,t) := \int K( |z-(x+iy)| )\varphi_j( x+iy) dxdy \ \ \ \ \textrm{if} \ |z| \geq r/2 . \]
That function is leafwise surharmonic. (Observe that by construction $\varphi_j^{reg}$ is $0$ if $U_j$ does not intersect $D$.)

Let $0<\eta<1/2$ such that $\hat{U}_j \simeq \mathbb D \times (\eta, 1-\eta)\subset U_j$   is still a covering of $M$. Let $\hat{\psi}_j: U_j\rightarrow [0,1]$ be a smooth function of the form $\hat{\psi}_j (z,t)= \hat{\psi}_j(t)$, which is equal to $1$ on $\hat{U}_j$ and vanishes on $\mathbb D\times \big( [0,\eta/2]\cup [1-\eta/2, 1] \big)$.  Then $\psi_j= \hat{\psi}_j / \sum_{j} \hat{\psi}_j$ is smooth on $U_j$ and we have $\sum_j \psi_j =1$ on $M$.  The function $\psi_j \varphi_j^{reg}: U_j \rightarrow \mathbb R$ is smooth, leafwise surharmonic and vanishes on a neighborhood of $\partial U_j$. We extend it by $0$ outside $U_j$, so that it becomes a smooth and leafwise surharmonic function $\psi_j \varphi_j^{reg}$ on $M$. Let us verify that $\varphi = \sum_j \psi_j \varphi_j ^{reg}$ is convenient. Let $U_k$ be a flow box intersecting $D$. In some neighborhood of $D$ in $U_k$, the function $\varphi - \log |s_k|$ is equal to 
$$\sum _j \psi_j (\varphi_j^{reg} -\log |s_k|) = \sum_j \psi_j (\log |s_j| -\log |s_k|) = \sum_j \psi_j \log |1_k/1_j| . $$
Hence it extends as a smooth leafwise surharmonic function on $U_k$. Since $\log |1_k| = \varphi -\log |s_k|$ is smooth, the metric $| \cdot |$ on $L$ can be defined by $|s| = e^\varphi$ as desired.  
\end{proof} 

\section{Hyperbolicity}

In this section, we prove that Levi-flats and minimal sets in surfaces of general type are hyperbolic, and deduce consequences for their topology. We remark that the results are also valid for immersed Levi-flats, but as in section~\ref{s:harmonic currents}, for simplicity we assume $M$ is embedded.  

\subsection{Leaves are hyperbolic} \label{ss:hyperbolicity}

\begin{prop}\label{hyphyp}
Let $M \subset S$ be a minimal set or a Levi-flat hypersurface of class $C^1$ in a surface of general type. Then $\FF$ is hyperbolic. In particular, the genus of any compact leaf (if it exists) satisfies $g \geq 2$. 
\end{prop}

Before giving the proof of theorem \ref{hyphyp} let us recall classical facts about minimal surfaces and surfaces of general type, see \cite{BPV}, \cite{F}, \cite{Ko} and subsection \ref{sgt1}. A $(-k)$-curve is a curve biholomorphic to $\mathbb P^1$ with self-intersection $-k$. A surface $S$ is \emph{minimal} if it does not have any $(-1)$-curve. Every non minimal surface can be obtained from a minimal one by performing finitely many blowups. A minimal surface of general type has finitely many $(-2)$-curves. Let $\mathcal C$ denote the union of these curves. For $n$ large enough,  the pluricanonical map $S \to \mathbb P^N$ associated to $\vert n K_S \vert$ contracts each curve of $\mathcal C$ to singular points and is biholomorphic elsewhere. In particular, by considering the pullback of the Fubini-Study metric, $K_S$ supports a metric with positive curvature outside $\mathcal C$.

\begin{proof}
Assume that there exists a compact leaf $C$ isomorphic to $\mathbb P^1$ (elliptic leaf).  The classical adjunction formula  $-2 = K_S \cdot C + C ^2$ then contradicts $K_S \geq 0$ and $C^2 = 0$ (see lemma \ref{normalbott}). Assume now that there exists a parabolic leaf $L$ and let $A$ be an Ahlfors current such that $\chi(A) =0$ (see proposition \ref{cha}). Using the leafwise adjunction formula
 \[   - \chi(A)  = A \cdot {K_S}_{\, \vert \FF} + A\cdot N  \] 
and the fact that $ A \cdot N = 0$ (see lemma \ref{normalbott}), we deduce $A \cdot {K_S}_{\, \vert \FF} = 0$. If $S$ is a minimal surface, then $K_S$ has a metric of positive curvature outside $\mathcal C$. Since the curves of $\mathcal C$ are not leaves of $\FF$ (they are elliptic and use the argument above, or their self intersection is $-2$ and use lemma \ref{normalbott}), the local structure of foliated cycle shows that  $A \cdot {K_S}_{\, \vert \FF} > 0$, contradiction. 

Now we focus on the non minimal case. It suffices to consider a blowup $\pi : S' \to S$ of a minimal surface of general type. By restricting the formula $K_{S'} = \pi^* K_{S} + E$ (see \cite{F}, section 3) to the foliation $\FF$, and intersecting with $A$, we obtain 
\begin{equation}\label{kol}
 A \cdot  {K_{S'}}_{\, \vert \FF}    =  A \cdot  (\pi^* K_S)_{\, \vert \FF} +  A \cdot  E_{\, \vert \FF}  . 
\end{equation}
Now $\pi^* K_{S}$ has a metric of positive curvature outside $E$ and the proper transforms (of the curves) of $\mathcal C$. Moreover, by lemma \ref{normalbott}, these curves are not leaves of $\FF$: we have $E^2 = -1$ for the exceptional divisor, and a blow up does not increase the self-intersection of the proper transforms of the curves of $\mathcal C$. Hence the first term of the right side of (\ref{kol}) is positive from the local structure of foliated cycles. The second term is non negative by \cite{Ghys2}, lemma 4.5 (or by proposition \ref{blowup1} for foliated cycles).
 \end{proof}
 
Observe that even if their CR foliation is hyperbolic, there exist Levi-flats in surface of general type carrying a transverse invariant measure. Indeed, one can consider pull-backs by ramified coverings of a flat $\mathbb P^1$-bundle  whose representation $\pi_1(\Sigma) \to \mathrm{PSL}(2,\mathbb C)$ has a dense image in $\mathrm{PSO}(2,\mathbb R)$, see section \ref{s:examples}. 

\subsection{Poincar\'e metric} We consider the Poincar\'e metric in the unit disc given by 
\[ ds ^2 =  { 4 \, \vert dz \vert ^2\over (1-\vert z \vert ^2)^2} . \]
It is complete of gaussian curvature\footnote{The gaussian curvature of  $ds^2 = \rho^2 \vert dz \vert^2$ is   $- \Delta_m \log \rho = - {1 \over \rho^2} \Delta \log \rho$.} $-1$. This metric being invariant by the group of biholomorphisms of the unit disc, it induces a metric on any hyperbolic Riemann surface. Hence, on a minimal set or a Levi-flat in a surface of general type, the Poincar\'e metric on each leaf defines a metric on the tangent bundle to the foliation. The following continuity/compactness result is due to Verjovsky and Candel. It will be used in proposition~\ref{contflow}:

\begin{thm}\cite{V,Candel}\label{compactness}
Assume that $\FF$ is hyperbolic. Then the Poincar\'e metric is continuous. Equivalently, the set $U(\FF)$ of uniformization mappings $\pi : \mathbb D \to M$ of some leaf,  endowed  with the topology of  uniform convergence, is compact. 
\end{thm}

\subsection{Exponential growth}\label{NOVI}

In this paragraph we first observe that Levi-flats in K\"ahler surfaces have trivial second fundamental group. We then prove that the fundamental group of Levi flats in surfaces of general type have exponential growth. In particular they do not carry the geometries of $\mathbb S^3$, $\mathbb S^2 \times \mathbb R$, $\mathbb R^3$, nor $\mathrm{Nil}$. Analogous results were obtained in~\cite{IM} for Levi-flats in the complex projective plane. We follow their argument, using in our context the hyperbolicity of the foliation (see proposition~\ref{hyphyp}) and a particular -- and presumably well-known -- case of Plante's theorem for codimension one foliations without transverse invariant measure (see \cite{Plante}, corollary 6.4). 

Novikov's theorems are central, let us recall the statements (see \cite{N} and \cite{CanCon2}, chapter 9). Let $M$ be a compact orientable $3$-manifold endowed with a transversely
orientable $2$-dimensional foliation $\FF$ of class $C^2$. If one of the conditions is satisfied
\begin{enumerate}
\item[(a)] there exists a leaf $L$ such that the inclusion map $\pi_1(L) \to \pi_1(M)$ has a non-trivial kernel,
\item[(b)] the second fundamental group $\pi_2(M)$ is not trivial
\end{enumerate}
then either the foliation $\FF$ has a Reeb component $\mathbb D \times \mathbb S^1$ or the $3$-manifold $M$ is diffeomorphic to $\mathbb S^2 \times \mathbb S^1$ and $\FF$ is the product foliation. Let us notice that (a) is equivalent to the existence of a Reeb component and to the existence of a vanishing cycle. 

In the context of K\"ahler surfaces, the CR foliation of a Levi-flat is taut (the K\"ahler form is positive on complex directions, see~\cite{S}),  in particular it has no Reeb component. The proof of this last fact is actually very simple: the homology class of any compact holomorphic curve can not be trivial by Stokes formula. We deduce that if $M$ is a $C^2$ Levi-flat in a K\"ahler surface, then either $M$ has trivial second fundamental group or $M$ is diffeomorphic to $\mathbb S^2 \times \mathbb S^1$. In the latter case the leaves have zero self-intersection, that implies that $S$ is a rational fibration (up to a modification, see \cite{BPV}, section V.4) and that $M$ is tangent to the fibration. 

Let us now focus on the exponential growth of the fundamental group. The following result shows in particular that Levi-flats in surfaces of general type are not diffeomorphic to the unitary tangent bundle of $\mathbb P^1$ or  an elliptic curve. That allows to assume $g \geq 2$ in the rigidity theorem~\ref{hypunit}.

\begin{prop} \label{t:polynomial Levi-flat}
Let $M$ be a $C^2$ Levi-flat hypersurface in a compact complex surface of general type. Then the fundamental group of $M$ has exponential growth. 
\end{prop} 

\begin{proof}
By proposition~\ref{hyphyp}, all leaves are hyperbolic Riemann surfaces. Because the Poincar\'e metric is continuous, the universal coverings of the leaves have exponential growth. Assume by contradiction that the fundamental group of $M$ has sub-exponential growth. By Novikov, the map $\pi_1(L) \to \pi_1(M)$ is injective for every leaf. Hence the pull-back foliation $\widetilde{\mathcal F}$ on the universal covering $\widetilde{M}$ has  simply connected leaves. Let us consider a finite covering $\{ U_i\}_{i\in I}$ of $M$ by foliated charts. Let $U_i^0$ be some lift of $U_i$ in $\widetilde{M}$. Observe that $\{ g U_i^0 \}$ for $g\in \pi_1(M)$ and $i\in I$ is a covering of $\widetilde{M}$ by foliated charts of $\widetilde{\mathcal F}$. Moreover, because $\pi_1(M)$ has subexponential growth $\delta$, the number of charts of this cover in $B_{\widetilde{M}}(x,R)$ is subexponential. By the pigeon hole principle, there exists $g U_i^0$ whose intersection with $L$ contains at least two different plaques. Hence there exists a simple closed loop in $\widetilde{M}$ which is transverse to $\widetilde{\mathcal F}$ (Novikov). Its projection in $M$ is a loop transverse to $\mathcal F$ which is homotopically trivial. But that must produce a Reeb component. \end{proof}

\section{Leafwise Brownian motion and Lyapunov exponent}\label{BMLE}

As in the preceeding sections the arguments work for immersed Levi-flats and minimals sets, for simplicity we assume that they are embedded. 

\subsection{Heat kernel and Brownian motion} 

Let $M$ be a minimal set or a Levi-flat hypersurface in a compact complex surface. Given a metric $m$ on the tangent bundle $T_\FF$, we denote by $vol_m$ the leafwise volume form and by $\Delta_m$ the leafwise Laplace-Beltrami operator. We have $2i \partial \bar \partial  = \Delta_m  \cdot vol_m$ and $2i \partial \bar \partial  = \Delta \cdot dx \wedge dy$ using usual laplacian. Let $p_m(t,x,y)$ be the leafwise heat kernel: this is the smallest positive solution of the heat equation
\[ {\partial \over \partial t} = 
 \Delta_m \ \ , \ \ \lim_{t \to 0} p_m(t,x,y) = \delta_x (y) \ , \]
where $\delta_x$ denotes the Dirac mass at $x$, the limit being in the sense of distributions (see \cite{chavel}). We consider the Brownian motion on $L_x$ with transition probability $p_m(t,x,y)$. This is a diffusion process with continous sample paths. Since the leaves are complete and of bounded geometry, the latter are defined for every $t \in \mathbb R ^+$. Let  $\Gamma_x := \{ \gamma : \mathbb R ^+ \to L_x \,  , \, \gamma(0) = x \}$ be the leafwise continuous paths, and $W_x$ be the Wiener probability measure on $\Gamma_x$ given by Brownian motion.

\subsection{Harmonic measure and Garnett's theory} 

We refer to \cite{Garnett}. A \emph{harmonic measure} is a probability measure $\mu$ on $M$ satisfying  
\[ \forall \psi \in A^0_\FF \ , \  \int  \Delta_m \, \psi  \, d\mu  = 0 . \]
There always exist harmonic measures, this can be proved using Hahn-Banach theorem or Markov-Kakutani fixed point theorem. The set of harmonic measures is convex, its extremal points are called \emph{ergodic}. Harmonic measures allows to introduce ergodic theory and  Lyapunov exponents. Let us specify the associated dynamical system $(\Gamma, (\sigma_t)_{t \geq 0} , \bar \mu)$. The phase space is $\Gamma := \cup_{x \in M} \Gamma_x$ and $(\sigma_t)_{t \geq 0}$ is the shift semi-group acting on $\Gamma$ by  
\[ \sigma_t (\gamma) (u) := \gamma (t+u) . \]
 The invariant probability measure $\bar \mu$ on $\Gamma$ is defined by \[ \bar \mu  := \int_M W_x \, d\mu(x)  .\]
 

\subsection{Lyapunov exponent}

For any $x \in M$ and   $\gamma \in \Gamma_x$,  let 
\[ h_{\gamma,t} : \tau_x \to \tau_{\gamma(t)} \] 
denote the holonomy map over the path $\gamma : [0,t] \to L_x$. Here $(\tau_x)_{x \in M}$ denotes a family of discs (or intervals if we deal with Levi-flats) transversal to $\FF$ and centered at $x$, it can be constructed from transversal discs in the foliated charts. To simplify the exposition, we  denote $h'_{\gamma,t}$ for $h'_{\gamma,t} (0)$. Let $\vert \cdot \vert$ be a metric on $N_\FF$.

\begin{thm} \label{limexplya}
Let $\mu$ be an ergodic harmonic measure on $M$. There exists $\lambda \in \mathbb R$ such that for $\mu$-almost every $x \in M$ and $W_x$-almost every $\gamma \in \Gamma_x$ 
\[ \lim_{t\to + \infty} \ \frac{1}{t} \log \ \lvert h_{\gamma,t} ' \rvert  = \lambda . \]
The number $\lambda$ is called the \emph{Lyapunov exponent} of $\mu$.
 \end{thm}

The proof uses Birkhoff's ergodic theorem for  $(\Gamma, (\sigma_t)_{t \geq 0} , \bar \mu)$. The compactness of $M$ ensures that the exponent is finite and does not depend on $\vert \cdot \vert$. However, it  depends on $m$ \emph{via} the Wiener measure $W_x$ on the leaves.

\subsection{Cohomological expression of the Lyapunov exponent}

The following observation gives the relation between harmonic measures and harmonic currents, see \cite{Ghys2} for more details.  

\begin{remark} \label{l:decomposition}
Given a  metric $m$ on the tangent bundle $T_\FF$, there is a bijection between the projective classes of harmonic currents on $M$ and the harmonic measures supported on $M$ given by $ \mu \mapsto T$, where $T$ is the unique current such that $\mu = T \wedge \mathrm{vol}_m$. 
\end{remark} 
 
Now we can state a fundamental formula. 

\begin{prop}\label{gogo} Let $M$ be a minimal set or a Levi-flat hypersurface in a compact complex surface $S$. Let $\mu$ be an ergodic harmonic measure and $T$ be the unique harmonic current such that $\mu = T \wedge \mathrm{vol}_m$. Then the Lyapunov exponent of $\mu$ for the metric $m$ is equal to 
\[  \lambda  =  - 2 \pi \ T \cdot N_\FF .  \]
\end{prop}

The proof is based on the formula \[ \lambda =  \int_M 
 \Delta_m \log \vert \cdot \vert  \,  d\mu, \] 
 where $\vert \cdot \vert$ is any metric on $N_\FF$ (see \cite{DK}, theorem 2.10). That expression comes from the heat equation  $\Delta_m = {\partial \over \partial t}$ and the ergodicity of the generator $\Delta_m$. 
The conclusion follows from the relation $ 2i \partial \overline{\partial} = \Delta_m \cdot \mathrm{vol}_m$. 

\subsection{Application} 

\begin{prop}  
\label{lyapgentype} 
Let  $M$ be a minimal set or a Levi-flat hypersurface of class $C^1$ in a surface of general type $S$, endowed with the Poincar\'e metric on the leaves.  Then the Lyapunov exponent of any ergodic harmonic measure $\mu$ for the Poincar\'e metric satisfies $\lambda > -1$.
\end{prop} 

\begin{proof}
Let $T$ be the unique harmonic current such that $\mu = T \wedge \mathrm{vol}_m$. Then, by proposition~\ref{gogo} and adjunction formula (proposition \ref{adju}), we have
 \[  \lambda = -2\pi \, T \cdot T_\FF +   2\pi  \,  T \cdot  {K_S}_{\, \vert \FF}  = -1 + 2\pi  \,  T \cdot  {K_S}_{\, \vert \FF}  \]
If $ds^2 = \rho^2 \vert dz \vert^2$ denote the Poincar\'e metric, then
\[ T \cdot T_\FF = T({1\over 2i\pi} \partial \bar \partial \log \rho^2) = T(  -{1\over 2\pi} \Delta_m  \log \rho  \cdot \mathrm{vol}_m) = - {1 \over 2\pi} , \]
since the gaussian curvature of the metric is equal to $-1$. It remains to prove that $T \cdot  {K_S}_{\, \vert \FF}$ is positive. If $S$ is a minimal surface, then $K_S$ has a metric of positive curvature outside $\mathcal C$,  and these curves are not leaves of $\FF$. That implies $T \cdot  {K_S}_{\, \vert \FF} > 0$ from the local structure of harmonic currents. For the non minimal case, we follow the arguments of theorem \ref{hyphyp}, replacing foliated cycles by harmonic currents. If $\pi : S' \to S$ is a blowup of a minimal surface of general type, then   
\begin{equation}\label{kol2}
 T \cdot  {K_{S'}}_{\, \vert \FF}    =  T \cdot  (\pi^* K_S)_{\, \vert \FF} +  T \cdot  E_{\, \vert \FF}  . 
\end{equation}
As before, $\pi^* K_{S}$ has a metric of positive curvature outside $E$ and the proper transforms of $\mathcal C$, which are not leaves of $\FF$. The first term of the right side of (\ref{kol2}) is positive from the local structure of harmonic currents. The second term is non negative by proposition \ref{blowup1}. \end{proof}

\subsection{Remark} To end this section, we recall a formula for the Lyapunov exponent on exceptional minimal sets in $\mathbb P^2$, see \cite{DK}. They are hyperbolic by \cite{CLS}, we endow the leaves with the Poincar\'e metric. 

\begin{prop}  \label{p: plane lyapunov exponent} 
Let $\mathcal F$ be a singular holomorphic foliation of $\mathbb P^2$ of degree $d \geq 2$. Let $M$ be a hypothetical  minimal set. Then the Lyapunov exponent of its harmonic measure for the Poincar\'e metric is equal to  \[  \lambda =  -    \frac{d+2}{d-1} . \]
\end{prop} 
The proof relies on $N_{\mathcal F} = \mathcal O(d+2)$ and $K_{\mathcal F} = \mathcal O(d-1)$, see \cite{Brunella1}. Then 
 \[ \lambda = -2\pi \,  T \cdot N_\FF = 2\pi \, {d+2 \over d-1} \, T \cdot T_\FF \]
and the expected value follows. Observe that the exponent of a holomorphic foliation of $\mathbb P^2$ is fixed by its degree. The situation is different for rational maps on $\mathbb P^1$, where the exponent of the maximal entropy measure depends on the map.

\section{Anosov Levi-flat}\label{anoso}

We call a Levi-flat \textit{Anosov} if its CR foliation is topologically conjugated to the weak stable foliation of an Anosov flow on a compact manifold. 

\begin{thm}\label{generesult}
Let $M$ be an Anosov Levi-flat. Then $\FF$ has no transverse invariant measure, in particular its CR foliation is hyperbolic. If we endow the leaves with the Poincar\'e metric then the Lyapunov exponent of any ergodic harmonic measure $\mu$ satisfies $\lambda \leq -1$.
\end{thm}

The absence of transverse invariant measure is based on the Anosov property. For the Lyapunov exponent, we prove that we can stretch the trajectories of the Anosov flow in the uniformizations, that allows to construct a continuous flow whose orbits are leafwise geodesics for the Poincar\'e metric. The resulting flow is called a \emph{point at infinity}. The bound on $\lambda$ is proved in that more general context, it relies on volume estimates in the spirit of Margulis-Ruelle's inequality.

\begin{cor}\label{c:no Anosov levi-flat}
A $C^1$ Levi-flat in a surface of general type is not Anosov. 
\end{cor}

\begin{proof}
We use Lyapunov exponents. Suppose that there exists such a Levi-flat, let $\mu$ be an ergodic harmonic measure and $\lambda$ be the Lyapunov exponent of $\mu$. On one hand, theorem \ref{lyapgentype} states that $\lambda > -1$ because the surface is of general type. On the other hand, theorem \ref{generesult} asserts that $\lambda \leq -1$. 
\end{proof}

We apply this result in section~\ref{s:sol PSL} to the weak stable foliation of the geodesic flow on the unit tangent bundle of negatively curved surfaces and to the weak stable foliation of the suspension of  Anosov diffeomorphisms of the $2$-torus (hyperbolic torus bundles). Theorem~\ref{generesult} is of independant interest since a lot of Anosov flows were constructed on $3$-manifolds, including hyperbolic or graph manifolds \cite{BF, FW, Goodman, HT}. We also point out that there is no requirement  on the CR structure of the foliation, this is important since there is an infinite number of moduli of CR structures~\cite{Dnonrigidity}. 

\subsection{Proof of theorem~\ref{generesult}}
We say that a Levi-flat $M$ has a \textit{point at infinity} if it is hyperbolic and if it supports a continuous flow $\psi : M\times \mathbb R \rightarrow M$ tangent to the CR foliation which lifts (on the universal covering $\mathbb D\rightarrow L$ of every leaf) to a flow whose trajectories are hyperbolic geodesics (parametrized by arc length) tending to the same point at infinity in $\partial \mathbb D$. 

\begin{lem} [Stretching] \label{l:anosov}
An Anosov Levi-flat has a point at infinity.
\end{lem}

The proof is postponed to  subsection~\ref{ss:stretching}. The arguments are based on the fact that trajectories of Anosov flows are quasi-geodesics, hence they have a well-defined limit in $\partial \mathbb D$ when lifted to the universal cover. By the Anosov property, this limit does not depend on the trajectory in the leaves. We then consider the flow $\psi$ defined by the geodesics pointing towards this point at infinity. Continuity of the limit point of quasi-geodesics for the compact open topology ensures the continuity of that flow. 

\begin{lem} [Volume estimates] \label{l:pt at infinity}
Let $M$ be a Levi-flat having a point at infinity. We endow the CR foliation with the Poincar\'e metric. Then the Lyapunov exponent of any ergodic harmonic measure satisfies $\lambda \leq -1$. 
\end{lem}

The proof is given in subsection~\ref{ss:margulis-ruelle}.  The flow $\psi$ induced by the point at infinity is absolutely continuous, and we relate its Jacobian with the Lyapunov exponent by shadowing geodesics with Brownian trajectories. The bound on the exponent follows from volume estimates in the spirit of Margulis-Ruelle inequality. 


\subsection{Stretching trajectories of an Anosov flow}\label{ss:stretching}
By assumption the foliation is topologically conjugated to the weak stable foliation of an Anosov flow. It is classical that the conjugation can be chosen smooth along the leaves. Hence, with no loss of generality, we can suppose until the end of subsection~\ref{ss:stretching} that $\FF$ is the weak stable foliation $\FF^s$ of an Anosov flow $A: M\times \mathbb R \rightarrow M$. We shall write $t \mapsto A_x(t)$ or $x \mapsto A^t(x)$ when $x$ and $t$ are respectively fixed.  

Let us first observe that $\mathcal F$ has no transverse invariant measure. Indeed, a leafwise volume form on $\mathcal F$ is uniformly exponentially contracted by the Anosov flow. Hence, if there exists a transverse invariant measure on $\mathcal F$, its product with a leafwise volume form is a finite measure which is uniformly and exponentially contracted by the Anosov flow. This contradicts that the total volume is preserved. As a consequence, the weak stable foliation has hyperbolic leaves.

Now we stretch the trajectories of the Anosov flow. Let us endow $T_{\FF}$ with the leafwise Poincar\'e metric. Let $x \in M$ and $\pi : \mathbb D \to L_x$ be a uniformization satisfying $\pi(0) = x$. We define
\[ \forall t \in \mathbb R \ , \ \alpha(t) := \pi^{-1} A_x(t)  ,  \]
which  is well defined by the covering property of $\pi$.

\begin{prop}\label{puol}
There exists a constant $\rho>0$ (independent of $\alpha$) satisfying for every  $(t,t') \in \mathbb R^2$:
\[ \rho^{-1}  \vert t - t' \vert   \leq  d_P ( \alpha(t) , \alpha(t') ) \leq \rho \vert t - t' \vert . \]
\end{prop}

\begin{proof} 
The upper bound is just the fact that the vector field $X$ is bounded for the Poincar\'e metric: one can take $\rho = \max_M  \vert X \vert$. Let us verify the lower bound.  Observe that the normal bundle to the weak stable foliation is identified with the strong stable foliation $E^{uu}$. The Anosov property yields a continuous (adapted) metric on this bundle and  $\xi > 0$ such that 
\[  \forall t \geq 0 \ , \  \forall v_u \in  E^{uu} \ , \ \vert dA^t(v_u) \vert \geq e^{\xi t} \vert v_u \vert .\]
In particular, if $h$ denotes the holonomy of the weak stable foliation, then 
\[ \vert h'_{\pi \circ \alpha|_{[t,t']}} \vert   \geq  e^{\xi (t'-t)} . \]
 On the other hand, there exists $\xi' > 0$ such that the derivative of a holonomy map along a path $\gamma$ is always $\leq e^{\xi' \mathrm{length} (\gamma)}$ (use local coordinates). Hence  $d_P(\alpha(t), \alpha(t')) \geq  (\xi - \xi' )(t'-t)$ as desired. 
\end{proof}

Proposition~\ref{puol} implies that the $\alpha$'s are quasi-geodesics. Such a property implies that there exists $\delta := \lim_{t \to \infty} \alpha(t) \in \partial  \mathbb D$. 
We define a vector field on $M$ by setting $Y(x) := d_0 \pi ( \delta )$, 
it does not depend on $\pi$ (whereas $\delta$ does). In other words, $Y(x)$ is the unitary direction at $x$ which looks at the limit of
the $A$-trajectory starting from $x$. 

Let $\psi(x,t)$ be the flow on $M$ associated with $Y$. By definition, the path 
\[ \forall t \in \mathbb R \ , \ \beta(t) := \pi^{-1} \psi_x(t)    \] 
coincides with the geodesic radius $[0,\delta[$. Observe that because of the Anosov property, if $\alpha'$ is the lift of another trajectory contained in the leaf through $x$, then the limit $\lim_{t\rightarrow \infty} \alpha'(t)$ is also $\delta$. Hence, the flow $\psi$ lifts to the universal cover of the leaves to a flow whose trajectories are geodesics tending to the same point at infinity. To finish the proof, we need to prove that $\psi$ is continuous.  We use the following classical proposition, its proof uses hyperbolic geometry in $\mathbb D$.

\begin{prop} \label{mieux}
Let $\alpha_n: \mathbb R^+ \rightarrow \mathbb D$ be a sequence of $\rho$-quasigeodesics starting at $\alpha_n(0)=0$ and converging uniformly on compact sets to a $\rho$-quasigeodesic $\alpha$. If  $\delta_n := \lim_{t\rightarrow +\infty} \alpha_n(t) $ and $\delta := \lim_{t\rightarrow +\infty} \alpha(t) $, then $\lim_n \delta_n = \delta$. \end{prop}

Let us deduce the continuity of $\psi$. It suffices to prove that the vector fields $Y$ is continuous on $M$, hence we want to show $Y(x_n) \to Y(x)$ for every $x_n \to x$. By compactness, we can assume $Y(x_n) \to Y$ up to take a subsequence.  Let $\pi_n : \mathbb D \rightarrow L_{x_n}$ be uniformizations. Taking again a subsequence, $\pi_n$ converges uniformly on compact sets to a uniformization $\pi : \mathbb D \rightarrow L_x$ (see theorem \ref{compactness}). The lift $\alpha_n := \pi_n^{-1} A_{x_n}$ also converges to the lift $\alpha := \pi^{-1} A_x$. Since these paths are $\rho$-quasigeodesics, there exist $\delta_n := \lim \alpha_n$, $\delta := \lim \alpha$, and we have $\delta_n \to \delta$ by proposition \ref{mieux}. Composing by $d_0 \pi_n$ and $d_0 \pi$, we obtain $Y =Y(x)$. Proof of lemma~\ref{l:anosov} is complete.

\subsection{Volume estimates}\label{ss:margulis-ruelle}
Let $\psi: M\times \mathbb R\rightarrow M$ be the continuous flow coming from the point at infinity. The next proposition shows that $\psi$ is absolutely continuous and it computes its Jacobian. Let $\nu:=\theta \otimes \textrm{vol}_P$, where $\theta$ is the volume measure of some fixed metric on $N_\FF$ and $\mathrm{vol}_P$ is the leafwise Poincar\'e volume.

\begin{prop}\label{contflow} $\psi$ is absolutely continuous with respect to $\nu$, and  
\[  \forall t \geq 0 \ , \ \forall x \in M \ , \  \frac{\psi^{t}_* \nu}{\nu}(x) = e^{t} \cdot \vert h'_{(\psi^{-s}(x))_{0\leq s\leq t}}  \vert .\]
\end{prop}

\begin{proof}
Let us compute the jacobian of $\psi^t$. It suffices to consider $t$ small. Every point is then moved in the same flow box, and the homeomorphim $\psi^t$ takes the form 
 \[ \psi^t (z,u) = (\varphi^t(z,u), u) , \]
where   $\varphi^t(.,u)$ is smooth and depends continuously on $u$ in the smooth topology. Writing $\theta = f(z,u) du$, elementary considerations show that $\psi^t$ is absolutely continuous and that
\[ \frac{\psi^{t} _* \nu}{\nu} = \frac{(\varphi^{-t})^*\textrm{vol}_P}{\textrm{vol}_P} \cdot \frac{f\circ \varphi^{-t}}{f} =  \frac{(\varphi^{-t})^*\textrm{vol}_P}{\textrm{vol}_P} \cdot \vert h' _{(\psi^{-s}(x))_{0 \leq s\leq t}} \vert . \]
We can uniformize any leaf by the upper half plane, in such a way that $\varphi^t(.,u)$ acts as $x+iy \mapsto x+ ie^t y$. Since $\textrm{vol}_P = dx \wedge dy/y^2$ in those coordinates, we have ${(\varphi^{-t})^*\textrm{vol}_P} = e^t {\textrm{vol}_P}$ as desired. 
\end{proof}
 
We now relate the Jacobian of $\psi$ to the Lyapunov exponent $\lambda$ of any ergodic harmonic measure $\mu$ (with respect to Poincar\'e metric). 

\begin{prop}\label{expogeo}
There exists a Borel set satisfying $\nu(B_\mu) > 0$ and 
\[  \forall x \in B_\mu \ , \ \lim_{t \to + \infty} {1 \over t} \log \vert h'_{(\psi^{-s}(x))_{0\leq s\leq t}}  \vert  =  \lambda . \]
\end{prop}


We shall need the three results below.  Let $W$ denote the Wiener measure on the set of continuous paths $\Gamma$ in $\mathbb D$ starting at the origin. 

\begin{prop}\label{kingman}
For $W$-almost every $\gamma \in \Gamma$, we have 
\[ \lim_{t\rightarrow \infty} {1 \over t} d_P( 0, \gamma(t)) = 1. \]
\end{prop} 

This result is classical (see e.g. \cite{Pinsky}, section 9.6), we provide a proof in an appendix (see section \ref{kingking}).  
We shall use the well-known shadowing property of geodesics by Brownian paths (see \cite{A}, th\'eor\`eme 7.3). Recall that $W$-a.e. $\gamma$ has a limit $\delta = \lim_{t\rightarrow \infty} \gamma(t) \in \partial \mathbb D$. 

\begin{thm}[Ancona]\label{ancona}
For $W$-a.e. $\gamma\in \Gamma$, there exists $c > 0$ such that \[    d_P ( \gamma(t) ,  [0,\delta[ )  \leq c \log t  \]
for $t$ large enough.
\end{thm}

Finally let us recall the following proposition from \cite{DK}.

\begin{prop} \label{DK2}
There exists an open set $U \subset M$ satisfying the following properties. For every $x \in U$, there exists $\Lambda_x \subset \Gamma_x$ such that 
\begin{enumerate}
\item $W_x(\Lambda_x) >0$,
\item $\forall \gamma \in \Lambda_x$, $\lim_{t \to + \infty} {1 \over t} \log \vert h'_{\gamma , t }  \vert  =  \lambda$. 
\end{enumerate}
\end{prop}

This result needs the H\"older regularity of the metric, counterexamples are given in~\cite{DV}. Here we can use the reinforcement~\cite{DNS1} of Verjovsky's theorem showing that Poincar\'e metric is indeed H\"older. An alternative is to apply \cite{DK} to a smooth hermitian metric on $T_\mathcal F$, and to use the conformal invariance of Brownian motion in dimension $2$. \\

We deduce the next lemma, which gives proposition \ref{expogeo}. Let $x\in U$ and $\pi: \mathbb D \rightarrow L_x$ be a uniformization such that $\pi(0)= x$. We identify $\Gamma$ with $\Gamma_x$  via these coordinates. Let $A\subset \Gamma\simeq \Gamma_x$ be the set of paths satisfying Ancona's theorem and $\delta_x := \delta (A\cap \Lambda_x)$. 

\begin{lem}\label{expexp}
Let $\delta \in \delta_x$ and $\beta : \mathbb R \to \mathbb D$ be a geodesic tending to $\delta$ at infinity. Then 
\begin{equation*}\label{beta}
   \lim_{t \to + \infty} {1 \over t} \log \vert h'_{\beta , t }  \vert  =  \lambda . 
\end{equation*}
\end{lem}

That yields for every $x \in U$ a subset of $\psi$-trajectories in $L_x$ of positive (Lebesgue) measure satisfying the limit. By Fubini's theorem, that limit is satisfied on a set of positive (Lebesgue) measure, and proposition~\ref{expogeo} follows.

\begin{proof} (lemma \ref{expexp}) Let $\epsilon > 0$. Using the three results above we get for every $\gamma \in A\cap \Lambda_x$ and $t \geq t(\gamma,\epsilon)$ a time $\tau_t \geq 0$ satisfying
\begin{equation*}\label{fog}
d_P \left( \gamma(t), \beta(\tau_t) \right) \leq c \log t \ \ \textrm{ and } \ \ \vert  \tau_t  - t   \vert \leq \epsilon t + c \log t 
\end{equation*}
and 
\begin{equation*}\label{fi}
  \lim_{t\to + \infty} \ \frac{1}{t} \log \ \lvert h_{\gamma,t} '  \rvert  = \lambda .
\end{equation*}
Observe that we can replace $[0,\delta[$ by $\beta$ in Ancona's estimates, since these geodesics tend to the same point at infinity: they are exponentially asymptotic.  Let $\tilde \gamma$ be the geodesic joining  $\gamma(t) , \beta(\tau_t)$ and let $\hat \gamma$ be the geodesic joining $\beta(\tau_t), \beta(t)$. Let $l_t$ be the length of  $\tilde \gamma$ (the length of $\hat \gamma$ is $\vert \tau_t - t  \vert$). Using the fact that  $\gamma + \tilde \gamma + \hat \gamma$ is homotopic to $\beta$, we have the chain rule
\begin{equation}\label{concat}
  h_{ \beta , t} '  =   h_{\hat \gamma , \vert \tau_t - t \vert } '  \cdot  h_{\tilde \gamma , l_t} '  \cdot h_{\gamma,t} ' .
  \end{equation}
Using similar arguments as in the proof of proposition~\ref{puol}, we obtain
\[  \vert \log \vert h_{\tilde \gamma , l_t} '   \vert  \, \vert  \leq C \cdot l_t \leq C \cdot c \log t , \]
\[  \vert   \log \vert h_{\hat  \gamma , \vert \tau_t - t \vert } '   \vert \, \vert   \leq C \cdot \vert \tau_t - t \vert \leq C \cdot ( \epsilon t + c \log t ).    \]
Taking logarithm and dividing by $t$ the chain rule  (\ref{concat}), we obtain  
\[  {\lambda - C \epsilon}  \leq  \liminf_{t \to + \infty} {1 \over t }  \log \vert h_{\beta,t} ' \vert  \leq \limsup_{t \to + \infty} {1 \over t }  \log \vert h_{\beta,t} ' \vert  \leq   {\lambda + C \epsilon} . \]
The conclusion follows by letting $\epsilon$ to zero. 
\end{proof}

\section{Levi-flats carrying the geometry of $\mathrm{Sol}$ or $\mathrm{PSL}(2,\mathbb R)$} \label{s:sol PSL}

This section is devoted to the following theorem. 

\begin {thm}\label{hypunit}
Let $M$ be one of the following $3$-manifolds: 
\begin{enumerate}
\item a hyperbolic torus bundle, 
\item the unitary tangent bundle $T^1\Sigma$ of a compact surface $\Sigma$.
\end{enumerate}
Then $M$ is not diffeomorphic to a $C^2$ Levi-flat in a surface of general type. 
\end{thm}
 
This  is a consequence of corollary \ref{c:no Anosov levi-flat} (dynamical part) and proposition \ref{RIG}  (topological part). The proof of \ref{RIG} relies on deep rigidity theorems by Thurston \cite{Thurston}, Ghys-Sergiescu \cite{GS} and Matsumoto \cite{Matsu}. We can assume for the case (2) that the genus of $\Sigma$ is $\geq 2$ by  proposition  \ref{t:polynomial Levi-flat}. 

\begin{prop}\label{RIG}
Let $M$ be a $C^2$ Levi-flat in a surface of general type. Suppose that it is diffeomorphic to either (1) or (2)-with $\Sigma$  of genus $\geq 2$. Then $M$ is an Anosov Levi-flat. 
\end{prop}

Our arguments actually prove that $M$ can not be immersed as a $C^2$ Levi-flat in a surface of general type. Hence $M$ can not be a finite covering of  such a Levi-flat, proving    that the geometries $\mathrm{Sol}$ and $\mathrm{PSL}(2,\mathbb R)$ do not appear.

\subsection{Hyperbolic torus bundle}

Let $\mathbb T^2 = \mathbb R^2/\Lambda$ be a torus. For any $A \in GL_2(\mathbb Z)$ let $T _A$ be the quotient of $\mathbb T^2 \times \mathbb R$ by the relation $(m,t) \sim (Am,t+1)$. That manifold is a \emph{torus bundle}  over the circle. 
It  is called \emph{hyperbolic} if  $A$ is hyperbolic, namely when the modulus of its trace is  $> 2$. In that case, $A$ has two distinct real eigenvectors with irrational slopes, these directions define the \emph{stable and unstable foliations} of $A$ on $\mathbb T^2$. In that case $T_A$ supports the geometry of $\mathrm{Sol}$. The \emph{suspensions of the stable and unstable foliations} are the foliations on $T_A$ 
respectively defined as the quotients of the product foliations on $\mathbb T^2 \times \mathbb R$. The family of maps $(m,t) \to (m, t+t')$ defines an Anosov flow on $T_A$. We use the following.

\begin{thm}[Ghys-Sergiescu, \cite{GS}] \label{GS} 
Let $M$ be an orientable hyperbolic torus bundle $T_A$ endowed with a transversally orientable $C^{k+2}$ foliation $\mathcal F$ ($0 \leq k \leq \infty$). If that foliation does not have any compact leaf, then it  is $C^{k}$-conjugated to the suspension of the stable foliation on $T_A$.
\end{thm}

Let us deduce case (1) of proposition \ref{RIG}. Let $M$ be a hyperbolic torus bundle $T_A$ and assume that it is diffeomorphic to a $C^2$ Levi-flat in a surface of general type. Its CR foliation $\FF$ is of class $C^2$ by \cite{BF}. Up to take a double covering, we can assume that $M$ is orientable and $\FF$ is transversally orientable. Then, in view of theorem \ref{GS}, it suffices to prove that $\FF$ has no compact leaf. Suppose to the contrary that $\mathcal F$ has a compact leaf, denoted $C$. It has genus $g \geq 2$ by the general type property (see theorem \ref{hyphyp}).  Then the natural map $\pi_1(C) \to \pi_1(M)$ is not injective, because $g \geq 2$ and the topology of a torus bundle $T_A$ is not rich enough (solvability of the Lie group). Novikov's theorem (see section \ref{NOVI}) then implies that $\FF$ has a compact leaf of genus $0$ or $1$, which is excluded by theorem \ref{hyphyp}.

\subsection{Circle fibrations} 

Let $\Sigma$ be a compact oriented surface of genus $g \geq 2$. Let $\pi_1(\Sigma)$ denote its fundamental group and $\text{Diff}^r_+(\mathbb S^1)$ be the group of orientation preserving $C^r$ diffeomorphisms of the circle ($0 \leq r \leq \infty$). Let $\rho : \pi_1(\Sigma) \to \text{Diff}^r_+(\mathbb S^1)$ be a representation and $eu(\rho)$ denote its Euler number (see \cite{Ghys3, Scott}). If $X$ denotes the quotient of $\mathbb H \times \mathbb S^1$ by the action 
\[  \gamma (p,\xi) = (\gamma p , \rho(\gamma) \xi) \ \ , \ \ \gamma \in \pi_1(\Sigma)  \]
then $eu(\rho)$ gives the amount of twisting of the $\mathbb S^1$-bundle $X$. Let $\FF$ be the quotient of the horizontal foliation of $\mathbb H \times \mathbb S^1$. The Milnor-Wood inequality states that $\vert eu(\rho) \vert \leq - \chi(\Sigma) =  2g-2$. The maximal value is obtained when $\rho$ embeds $\pi_1(\Sigma)$ as a discrete subgroup of $\text{PSL}(2,\mathbb R)$, i.e. when $X$ is the unitary tangent bundle $T^1 \Sigma$. A converse is true in class $C^2$, see \cite{Matsu}:

\begin {thm} [Matsumoto] \label{matsu}
Let  $\rho, \rho' : \pi_1(\Sigma) \to \text{Diff}^2_+ (\mathbb S^1)$ be representations satisfying $eu(\rho) = eu(\rho') \in \{ - \chi(\Sigma) , \chi(\Sigma) \}$. Then $\rho$ and $\rho'$ are topologically conjugated.
\end{thm}

In other words, if $\vert eu(\rho) \vert$ is maximal, then $X$ is homeomorphic to $T^1\Sigma$ and $\FF$ can be topologically conjugated to the weak stable foliation of the geodesic flow for \emph{any} metric of negative curvature on $\Sigma$ (in particular the Poincar\'e metric). \\

Let us prove case (2) of proposition \ref{RIG}. Let $M = T^1\Sigma$ be the unitary tangent bundle of a compact   surface of genus $g \geq 2$. Assume it is diffeomorphic to a $C^2$ Levi-flat in a surface of general type, let $\FF$ be its CR foliation. 
 Theorem \ref{hyphyp} ensures that $\FF$ has no compact leaf of genus $1$. Moreover, it has no compact leaf of genus $\geq 2$. Indeed, the unitary tangent bundle to a hyperbolic compact closed surface do not contain incompressible surfaces of genus $\geq 2$, and by Novikov's theorem, a compact leaf is incompressible, since the CR foliation is taut. Hence, by Thurston's argument (see \cite{Thurston}, \cite{Lev}), there exists an isotopy between $\FF$ and a foliation transverse to the circle fibration over $\Sigma$. In particular, the foliation $\FF$ is a suspension given by a morphism $\rho : \pi_1(\Sigma) \to \text{Diff}^2_+ (\mathbb S^1)$. Moreover, the Euler number  $eu(\rho)$ is maximal, equal to $\chi(\Sigma)$. Theorem \ref{matsu} then implies that $\FF$ is topologically conjugated to the weak stable foliation of the geodesic flow for the Poincar\'e metric on $\Sigma$, as desired.

\subsection{Remark} 

Let $M$ be a (hypothetical) smooth Levi-flat in $\mathbb P^2$. By the classical results of section \ref{NOVI}, $M$ supports neither the geometry of $\mathbb S^3$ and $\mathbb S^2 \times \mathbb S^1$ (K\"ahler surface) nor of $\mathbb R^3$ and $\mathrm{Nil}$ ($M$ does not have any transverse invariant measure, see \cite{IM}). We claim that an analytic Levi-flat in $\mathbb P^2$ can not support the geometry of $\mathrm{Sol}$ or $\mathrm{PSL}(2,\mathbb R)$. First, by using \cite{Ghys3, GS} and the same arguments as before, the CR foliation of $M$ has a transverse projective structure. Now, if $M$ is analytic, then every connected component of $\mathbb P^2 \setminus M$ is Stein, and its CR foliation extends to a singular holomorphic foliation of $\mathbb P^2$ \cite{LN}. In our case it extends to a transversally projective singular holomorphic foliation on $\mathbb P^2$. But such a  foliation has an invariant algebraic curve, which is incompatible with the existence of a minimal set~\cite{CLS}.

\section{Hausdorff dimension of minimal sets}

Let $M$ be a minimal set in a complex compact surface $S$. Its holonomy pseudo-group is \textit{not discrete} if there exist transversal discs $\tau'  \subset \subset \tau$ and a sequence of holonomy maps $h_n : \tau' \rightarrow \tau$ different from identity which converges uniformly to identity.  Otherwise, it is called \textit{discrete}. We say that $M$ is discrete or not.  That notion of discreteness  was  introduced in \cite{Ghys} and \cite{Nakai}. Let us recall that if $\FF$ is minimal, then either there exists a transverse invariant measure or there exists a unique harmonic measure and its Lyapunov exponent is negative  \cite{DK}.

\begin{thm} \label{t: dimh}
Let $M$ be a  minimal set in a complex compact surface $S$. Assume that $M$ is discrete and does not support any transverse invariant measure. Let $m$ be a metric on the tangent bundle $T_\FF$ and $\mu$ be the unique harmonic measure. Then 
\[ \dim_{H} M \geq \frac{ h } { |\lambda| }  , \]
where $h$ denotes the Kaimanovich entropy of $m$, and $\lambda$ denotes the Lyapunov exponent of $\mu$.
\end{thm}

In the context of random walks on linear groups, the dimension of the harmonic measure is equal to entropy/exponent \cite{L2, L4}. Such a ratio also appears  for the dimension of invariant measures for dynamical systems, see e.g. \cite{Mane} for rational maps  on $\mathbb P^1$. It is possible to establish an analogous formula in our context, however we only need a lower estimate for the Hausdorff dimension.

\subsection{Kaimanovich entropy}  

We refer to \cite{K2}, \cite{K}, \cite{L3}. Let $M$ be a  minimal set and $m$ be a metric on $T_\FF$.  The \emph{Kaimanovich entropy} is defined by the following theorem. Recall that $p_m(x,y,t)$ denotes the heat kernel on the leaf $L_{x}$.

\begin{thm}
There exists $h_m \geq 0$ such that for $\mu$-a.e. $x\in M$, 
\begin{equation} \label{eq: entropy} h_m  = \lim _{t\rightarrow \infty} -\  \frac{1}{t} \int _{L_x} p_m(x,y,t) \log p_m(x,y,t) \mathrm{vol}_m (dy) .  
\end{equation}
Moreover, if $M$ has no transverse invariant measure, then $h_m >0$. 
\end{thm}

If $\mu$-almost every leaf is the unit disc and $m$ is the leafwise Poincar\'e metric of curvature $-1$, then $h_m = 1$. We shall use the following geometrical interpretation of the entropy: $h_m$ is the exponential growth of separated positions for Brownian motion. To be more precise, introduce the following definition. A subset $Z \subset L_x$ is a \emph{$(C,D)$-lattice} if its points are $C$-separated and if any point of $L_x$ lies at a distance $\leq D$ from $Z$. We define below the sequence of measures $(M_n)_n$ on the lattice $Z$.

\begin{prop} \label{l: signification of entropy}
 Let $x \in M$ generic for $\mu$. Let $Z$ be a lattice in $L_x$ and $(Z_n)_n$ be subsets of $Z$ such that $M_{n} (Z_n) \geq 1/2$. Then
\[ \liminf_{n\rightarrow \infty} \frac{1}{n} \log \  \lvert Z_n \rvert \geq h . \] 
\end{prop} 

To define $M_n$ we introduce $T_z  : = \cap_{z' \in Z} \  \{   d(y,z) \leq d(y,z')   \}$, these are the tiles of the Voronoi tesselation associated to the lattice.
Let $\pi$ be the projection $L_x \rightarrow Z$  (defined a.e.) and $B_n$ be the distribution of Brownian motion starting at $x$ at the time $n$ (with density $ p_m (x,y,n) \mathrm{vol}_m (dy)$). Then $M_n$ stands for the probability measure  $\pi_* B_n$. As announced, proposition \ref{l: signification of entropy} can be made more precise: for any lattice $Z$ in a generic leaf, there exists $(Z_n)_n$ such that $M_n(Z_n) \geq 1/2$ and $\frac{1}{n} \log \  \lvert Z_n \rvert \to h$.

\subsection{Proof of theorem \ref{t: dimh}} 

Let $M$ be a discrete  minimal set with no transverse invariant measure.  We fix $\varepsilon >0$, $x\in M$   and $\tau_x$ a transversal to $\FF$ at $x$. It suffices to prove that $\dim_{H} M \cap \tau_x \geq  \frac{h -3\varepsilon}{|\lambda|+3\varepsilon}$. Let $\mathbb D(r) \subset \tau_x$ denote the disc centered at zero of radius $r$, we put a subscript to indicate a center. By Koebe distortion theorem, there exists $\kappa> 0$ such that 
\[ \mathbb D_{h(0)} (e^{-\kappa} |h'(0)| \rho ) \subset h (\mathbb D (\rho) ) \subset \mathbb D_{h(0)} (e^{\kappa} |h'(0)| \rho ) \] 
for any holomorphic injective function defined on $\mathbb D(2 \rho)$. We need to control the domain of definition of holonomy maps. Proposition 2.1 of \cite{DK} asserts that for $\mu$-a.e. $x \in M$ and $W_x$-a.e. $\gamma \in \Gamma_x$, there exist $r_\gamma (\epsilon) >0$ and $t_\gamma (\epsilon)\geq 0$ such that for every $r \leq r_\gamma(\epsilon)$ and $t \geq t_\gamma(\epsilon)$:
\begin{enumerate}
\item[-] $h_{\gamma,t}$ is well defined from $\mathbb D  (r)  \subset \tau_x$ to $\tau_{\gamma(t)}$,
\item[-] $h_{\gamma,t}$ satisfies $e^{t (\lambda - \epsilon) } \leq \lvert h'_{\gamma,t}  (0) \rvert \leq e^{t (\lambda + \epsilon) }$.
\end{enumerate}
Let us fix $x$ generic for $\mu$, $r_0$ small and $t_0$ large  such that
\[ \Gamma_{\epsilon} := \{ \gamma \in \Gamma_x \ , \ r_\gamma(\epsilon) \geq 2 r_0 \, , \, t_\gamma(\epsilon) \leq t_0 \} \]
satisfies $W_x (\Gamma_{\epsilon}) \geq 1/2$.  Let  $r:= r_0 e^{-2\kappa} / 8$ ($\kappa$ is the distortion constant). The next proposition is proved in subsection \ref{ss:pseparated}.

\begin{prop}\label{p: separated}  For $n$ large enough there exist
\begin{enumerate}
\item[(i)] a set $I_n' \subset L_x \cap \tau_x$ of cardinality larger than $e^{n(h-3\epsilon)}$,
\item[(ii)] a   negative number $\lambda_n \in [n (\lambda-2\epsilon),n(\lambda + 2\epsilon)]$
\end{enumerate}
such that for every $z \in I_n'$, 
\begin{enumerate}
\item[(1)] $\mathbb D_z(r e^{\lambda_n} e^{-\kappa})  \subset  h_{x,z} (\mathbb D(r)) \subset \mathbb D_z (r e^{\lambda_n+\log 2} e^\kappa)$,
\item[(2)] $\{ h_{x,z}(\mathbb D(r)) \}_{z \in I_n'} \subset \tau_x$ are pairwise disjoint.
\end{enumerate}
\end{prop}

Let us see how theorem \ref{t: dimh} follows. Let $m \geq 1$ and
\[  \mathcal H_m := \{ h_{x,z_1} \circ \ldots \circ h_{x,z_m} \ , \   z_i \in I'_n\}  . \]
Next we introduce 
\[ K_m := \bigcup_{z \in \{ h(0) , \, h \in \mathcal H_m \} } \,  \mathbb D_z \left( r (e^{\lambda_n} e^{- 2\kappa})^m \right) , \]
these discs are pairwise disjoint and $K_m \subset \bigcup_{h \in \mathcal H_m} h (\mathbb D(r))$ by distortion. The family $\mathcal H_m$ forms an iterated function system. The limit set $K := \bigcap_m K_m$  satisfies \[ \dim _H  K \geq \frac{n(h-3\varepsilon)}{- \lambda_n  +  2\kappa } \geq \frac{h-3\varepsilon}{-\lambda + 3 \epsilon} , \]
where the second inequality relies on $\lambda_n \geq n (\lambda-2\epsilon)$ and $n \epsilon \geq 2\kappa$. That estimate holds for the Hausdorff dimension of $M \cap \tau_x$ since it contains $K$.
 
\subsection{Proof of proposition \ref{p: separated}} \label{ss:pseparated}

We work for $n$ large enough. The definition of $r$ is given before proposition \ref{p: separated}. First we show that there exist a set $I_n$ of cardinality larger than $e^{n (h-2\varepsilon)}$ and a negative number $\lambda_n$ satisfying $(ii)$ such that for any $z \in I_n$: 
\begin{enumerate}
\item[(a)]  $h_{x,z}$ is defined from $\mathbb D(2r_0) \subset \tau_x$ to $\tau_x$,
\item[(b)] $\lambda_n \ \leq \ \log \lvert h'_{x,z}  (0) \lvert \ \leq \ \lambda_n + \log 2$,
\item[(c)] $\mathbb D_z(r e^{\lambda_n} e^{-\kappa})  \subset  h_{x,z} (\mathbb D(r)) \subset \mathbb D_z (r e^{\lambda_n+\log 2} e^\kappa)$ (which is (1)),
\item[(d)] $h_{x,z} (\mathbb D(r)) \subset \mathbb D(r)$.
\end{enumerate}

Let $Z := L_x \cap \mathbb D (r/2)$.  By compactness and minimality of $M$, the subset $Z$ is a $(C,D)$-lattice. Recall that $\pi$ is the projection $L_x \rightarrow Z$ and $M_t = \pi_* B_t$, where $B_t$ is the distribution of Brownian motion at the time $t$. We set $\Gamma_{\varepsilon} (n) := \{  \gamma(n) \, , \, \gamma \in \Gamma_\epsilon \}$ and $J_n := \pi (\Gamma_\varepsilon(n) )$. Since $M_n ( J_n ) \geq 1/2$  by definition of $\Gamma_\varepsilon$,  proposition \ref{l: signification of entropy} yields $\lvert J_n \rvert \geq e^{n (h-\varepsilon) }$.

By compactness of $M$, there exist $E>0$ and $r_1 >0$ such that if $\gamma\subset L_{y}$ is a compact smooth path starting at $y$ of length bounded by $D$, then the holonomy map $h_\gamma$  is defined on $\mathbb D (r_1)$ and satisfies $- E \leq \log \lvert h'_{\gamma} (y) \rvert \leq E$.

By construction, every point $z \in J_n$ is $D$-distant from a point $y \in \Gamma_\varepsilon (n)$. 
By definition of $\Gamma_\epsilon$, $r_\gamma (\epsilon)$, $t_\gamma (\epsilon)$ and distortion theorem, $h_{x,y}$ is well-defined on $\mathbb D(2r_0)$ and its image is  contained in a disc of radius $2 r_0 e^{n (\lambda + \varepsilon)} e^\kappa$. Since that radius is smaller than $r_1$, the map $h_{x,z} = h_{y,z} \circ h_{x,y}$ is well-defined on $\mathbb D(2 r_0)$ and satisfies
\[ -E + n (\lambda -\varepsilon )  \leq \log \lvert h_{x,z} '(0) \rvert \leq E + n (\lambda + \varepsilon ) .\]
That implies $(a)$. That moreover implies that there exist  $\lambda_n$ satisfying 
 $$-E + n (\lambda -\varepsilon )  \leq \lambda_n \leq  E + n (\lambda + \varepsilon )$$ 
 and a subset $I_n\subset J_n$ of cardinality 
\begin{equation*}\label{ent} 
| I_n| \geq \frac{|J_n|\log 2} { 2 (E + \varepsilon n)}
\end{equation*}
such that $(b)$ is satisfied for every $z \in I_n$. The expected bounds $| I_n| \geq e^{n (h-2\varepsilon)}$ and (ii) for $\lambda_n$ follow by taking $n$ large. Koebe theorem ensures $(c)$ (which is (1) of proposition \ref{p: separated}), and we get $(d)$  from $h_{x,z}(0) \in \mathbb D(r/2)$ and $\text{diam} \, h_{x,z} (\mathbb D(r)) \leq r e^{\lambda_n+\log 2}e^\kappa$. \\ 

It remains to build $I_n' \subset I_n$ satisfying $(i)$ and $(2)$. 
Let $\kappa' := 2\kappa + \log 2$ and   
$$\mathcal H := \{ h: \mathbb D(r) \rightarrow \tau_x \textrm{ holonomy map } , \,    - \kappa ' \leq \log |h'(0)| \leq  \kappa ' \} . $$ 
That set is finite because the foliation is discrete. Let $r_n := 2 r e^{\lambda_n + \log 2} e^\kappa$ and $(z,z') \in I_n \times I_n$ satisfying  $|z - z'| \leq r_n$. Koebe   theorem yields \[ h_{x,z} \, (\mathbb D( 2r ) )\subset  \mathbb D_z (r_n) ,\]
\[ h_{x,z'} \, (\mathbb D( r_0 ) )\supset  \mathbb D_{z'} (r_0 e^{\lambda_n} e^{-\kappa}) = \mathbb D_{z'} ( 2 r_n)  \supset \mathbb D_z (r_n). \]
We deduce that $h := h_{x,z'} ^{-1} \circ h_{x,z} : \mathbb D(2r) \to \mathbb D(r_0)$ is well defined. One easily verifies that  $- \kappa ' \leq \log |h'(0)| \leq  \kappa '$, hence the restriction $h : \mathbb D(r) \to \mathbb D (r_0)$ belongs to  $\mathcal H$. Let $I_n' \subset I_n$ such that $\{ h_{x,z} \}_{z \in I_n'}$ represents the classes of $I_n$ modulo $\mathcal H$ by right composition. That set has cardinality 
$$|I_n'| \geq |I_n| / \vert \mathcal H \vert  \geq e^{n (h-2\varepsilon)} / \vert \mathcal H \vert \geq e^{n (h-3\varepsilon)} .  $$ 
The fact that   $\{ h_{x,z}(\mathbb D(r)) \}_{z \in I'_n}$ are disjoint follows from $\text{diam} \, h_{x,z} (\mathbb D(r)) \leq r_n / 2$ for $z \in I_n$ and from $|z - z'| > r_n$ for points of $I'_n$.

\subsection{Applications of theorem \ref{t: dimh}} 

\subsubsection{Proof of theorem~\ref{t:non discrete}}

We use the following lemmas.
\begin{lem} \label{t: irregular}
Let $S$ be a foliated complex surface and $M$ a minimal set with no transverse invariant measure. If $M$ is not discrete then it is an analytic Levi-flat. 
\end{lem} 

\begin{proof} Assuming that $M$ has no transverse invariant measure and  is not discrete,  there exists a hyperbolic fixed point in $M$ \cite[corollary 1.3]{DK}. That implies by \cite[section 4]{LR} that there are local holomorphic flows  in the closure of the pseudo-group. The minimal set is then analytic.  \end{proof} 

In surfaces of general type, theorem \ref{t: dimh} allows to prove a converse (see lemma \ref{t: first} below) when a leaf is simply connected. We begin with the 

\begin{lem} \label{simplyco}
Let $S$ be K\" ahler surface, $\FF$ a holomorphic foliation and $M$ a minimal set. If $M$ has a simply connected leaf, then every leaf is simply connected except for a countable number. 
\end{lem} 

\begin{proof}
First observe that the K\"ahler property implies that there is no vanishing cycle \cite{Ivash}. Then, by minimality and simply connected assumption, the holonomy of a leafwise (homotopically) non trivial loop is not the identity. Indeed, any deformation of that loop in a close leaf remains non trivial, and by minimality the simply connected leaf must contain such a deformation.

Now suppose that there are uncountably many leaves which are not simply connected. Consider a finite covering of $M$ by flow boxes. Let us fix in each non simply connected leaf a non trivial loop. Each of them induces a cyclic sequence of crossed flow boxes. By cardinality, there exists a cyclic sequence attained by uncountably many loops. Since the domain of definition of the corresponding holonomy map is open, one of its connected component has to intersect uncountably many loops. The holonomy map is therefore equal to identity on the component, which has been excluded. 
\end{proof}

\begin{lem} \label{t: first}
Let $S$ be a surface of general type, $\FF$ a holomorphic foliation and $M$ a minimal set with no transverse invariant measure. If $M$ has a simply connected leaf and is discrete then $\dim_H M > 3$.
\end{lem} 

\begin{proof}
We endow the leaves with the Poincar\'e metric $m$. Let $\mu$ be the unique harmonic measure and $\lambda$ its Lyapunov exponent. The measure $\mu$ has no atom since it is not transversely invariant. By lemma \ref{simplyco}, $\mu$-a.e. leaf is isomorphic to the disc, hence the Kaimanovich entropy is then equal to $1$. Proposition \ref{lyapgentype}  and theorem \ref{t: dimh} finally imply that the transverse Hausdorff dimension of $M$ is larger than $h_m / \vert \lambda \vert >1$.  
\end{proof} 

\subsubsection{Margulis-Ruelle type inequalities}

From $\dim_H M \cap T\leq 2$ for minimal sets and $\dim_H M \cap T = 1$ for Levi-flats, we obtain:

\begin{cor} \label{c: MR}
We take the assumptions of theorem \ref{t: dimh}. The Lyapunov exponent of the harmonic measure satisfies $\lambda \leq -1/2$. If   $M$ is Levi-flat, then $\lambda \leq -1$.
\end{cor}

That corollary provides an alternative proof of theorem~\ref{generesult} when $M$ is the unitary tangent bundle of a compact surface of genus $g \geq 2$. However, such a method does not work for a hyperbolic torus bundle, because its holonomy pseudo-group is not discrete: it contains an irrational rotation. 

\subsubsection{Dimension of  minimal sets in $\mathbb P^2$}\label{p2}

  Theorem \ref{t: dimh} implies the following lower bound for the transverse  dimension if there is a simply connected leaf. 
  
\begin{prop} Let $M$ be a minimal set of a holomorphic foliation of $\mathbb P^2$ of degree $d \geq 2$. If $M$ has a simply connected leaf, then 
\[ \dim_H M \cap T \geq \frac{d-1}{d+2}. \]
\end{prop} 

Indeed, if $M$ is not discrete, its transverse dimension is $1$ by lemma \ref{t: irregular}. If $M$ is discrete, its transverse dimension is larger than $1/\vert \lambda \vert$, and we use proposition \ref{p: plane lyapunov exponent}. Note that the conformal dimension of any conformal harmonic current is bounded above by the same constant \cite{Brunella}. Note also that these sets are not too small in the sense of potential theory: their harmonic current has finite energy \cite{FS1}.

\section{Appendix}  \label{kingking}

For reader's convenience we prove the classical propositions \ref{kingman} (drift of Brownian motion on $\mathbb D$). We propose arguments adapted to the unit disc, in particular we do not use the law of iterated logarithm, see \cite{Pinsky}. Recall that $\mathbb D$ is endowed with the Poincar\'e metric of gaussian curvature $-1$ and that we work with the heat equation ${\partial \over \partial t} = \Delta_m$. Let $W$ denote the  Wiener measure on the set of  continuous paths starting at the origin. By applying Kingman's   theorem to  $(\Gamma , (\sigma_t)_{t \geq 0}, W)$ and   $H_t(\gamma) :=d_P(\gamma(0), \gamma(t))$, there exists $\delta \in \mathbb R$ such that for $W$-almost every $\gamma \in \Gamma$:
\[ \lim_{t\rightarrow \infty} {1 \over t} d_P( 0, \gamma(t)) = \delta. \]
We want to show that $\delta =1$. 
Let  $\varphi : \mathbb D \to \mathbb R^+$ defined by $\varphi (z) = \frac{1-|z|^2}{|1 - z|^2}$. This is a harmonic function (it corresponds to the imaginary part  in the hyperbolic plane) and the laplacian of $\log \varphi$ with respect to the Poincar\'e metric is identically $-1$. Dynkin's formula asserts that 
\[ \forall t \geq 0 \ , \ \mathbb E \, (\log \varphi (\gamma(t) )   =   \log \varphi (0) +  \, \mathbb E \left( \int_0^t \Delta_m \log \varphi (\gamma(s) ) \, ds \right)  , \]  
where the expectation $\mathbb E$ holds on $\Gamma$ with respect to $W$. We deduce 
\begin{equation*}\label{lapla}
 \forall t > 0 \ , \  {1 \over t} \, \mathbb E \,(\log \varphi (\gamma(t) )) = -1 .
\end{equation*}
To complete the proof it suffices to show
\begin{equation}\label{newnew}
\lim_{t\rightarrow + \infty} {1 \over t} \, \mathbb E \,( \log \varphi (\gamma (t) ) ) =  - \delta . 
\end{equation}
Given a circle centered at the origin, its hyperbolic radius $R \in \mathbb R^+$ is related to its euclidian radius $r \in [0,1[$ by $r(R) = \frac{e^{R} -1 } {e^{R} +1}$.
For every $t \geq 0$, let $\rho_t$ be the pushforward measure of $W$ on $[0,1[$ by the map  $\gamma \mapsto r \circ d_P(0,\gamma(t))$.
Then $\mathbb E ( \log \varphi (\gamma (t) ) )$ is equal to
\begin{equation}\label{newnew2}
  \int_{[0,1[}  {1 \over 2\pi}  \int_0^{2\pi} \log \varphi (r e^{i\theta}) d\theta \ d\rho_t(r)  = :  \int_{[0,1[} M(r)  \, d\rho_t(r)  . 
\end{equation}
Using $r : \mathbb R^+ \to [0,1[$, we define
\begin{equation}\label{oupi}
 I_{t,\epsilon} :=   [ \, r( t(\delta-\epsilon) ) \, ,  \,  r( t(\delta+\epsilon) ) \,   ] .
 \end{equation}
We have $\rho_t(I_{t,\epsilon}) \geq 1-\epsilon$ for $t$ large enough, since $\lim_{t\rightarrow \infty} {1 \over t} d_P( 0, \gamma(t)) = \delta$ for almost every $\gamma$. We shall decompose the integral of (\ref{newnew2}) (right side) over $I_{t,\epsilon}$ and $[0,1[ \setminus I_{t,\epsilon}$. For that purpose observe that $M(r) = \log (1-r^2)$ from the classical formula $\frac{1}{2\pi} \int_0 ^{2\pi} \log |1- r e^{i\theta}|^2 d\theta = 0$. Now, from  
 \begin{equation*}\label{oup}
M(r) = \log (1 - r(R)^2) = \log \cosh ^{-2} (R/2) \in [-R, - R + \log 4 ] 
\end{equation*}
 and the definition (\ref{oupi}) of $I_{t,\epsilon}$,  we get 
\[ - t (\delta +  \epsilon)  \leq \int_{I_{t,\epsilon}}  M(r) \, d\rho_t(r)   \leq - t (\delta - \epsilon)  +  \log 4 . \]
Since $\rho_t$ is absolutely continuous with respect to Lebesgue measure on $[0,1]$ (from the heat kernel) and $M \in L^2[0,1]$, Cauchy-Schwarz inequality yields
\[  \int_{[0,1[ \setminus I_{t,\epsilon}}  M(r) \, d\rho_t(r) \leq  \epsilon \, \int_{[0,1]} M^2(r) \, d\rho_t(r) =: c \, \epsilon . \]
That completes the proof of (\ref{newnew}), and $\delta = 1$ as desired.

\vspace{1cm} 

\begin{small}

\noindent Universit\'e Paris-Sud and CNRS UMR 8628\\
B\^atiment 425, 91405 Orsay cedex, France.\\

\vspace{0.2cm}

\noindent e-mail: bertrand.deroin@math.u-psud.fr, christophe.dupont@math.u-psud.fr

\end{small}

\end{document}